\documentclass[a4paper,10pt]{article}

\usepackage[utf8]{inputenc}
\usepackage[T1]{fontenc}

\usepackage[sort,nocompress]{cite}
\usepackage{amsmath,amsfonts,amssymb,amsthm,cases,graphicx,authblk}
\usepackage{mathrsfs}
\usepackage{tikz}
\usepackage{tikz-3dplot}
\usepackage{pgfplots}
\usepackage{hyperref}
\hypersetup{colorlinks=true,linkcolor=teal,citecolor=blue,filecolor=teal,urlcolor=blue}
\usepackage{color}
\usepackage{vmargin}
\usepackage{nccmath}
\usepackage{verbatim}
\usepackage{float}
\usepackage{url}
\usepackage[capitalise]{cleveref}
\usepackage{fancyhdr}

\setmarginsrb{2.3cm}{1cm}{2.3cm}{1.5cm}{1cm}{1cm}{1cm}{1cm}

\renewcommand{\leq}{\leqslant}
\renewcommand{\le}{\leqslant}
\renewcommand{\geq}{\geqslant}

\everymath{\displaystyle}
\numberwithin{equation}{section}

\newtheorem{Theorem}{Theorem}[section]

\newtheorem{Proposition}[Theorem]{Proposition}
\newtheorem{Lemma}[Theorem]{Lemma}
\newtheorem{Remark}[Theorem]{Remark}

\newcounter{epsi} 
 
\title{Maximal regularity for a compressible fluid-structure interaction system with Navier-slip boundary conditions}

\author[1]{Kuntal Bhandari \thanks{E-mail Id: \href{mailto: bhandari@math.cas.cz}{bhandari@math.cas.cz}.}}

\author[2]{Imene Aicha Djebour
	\thanks{E-mail Id: \href{mailto: imene.djebour@cyu.fr}{imene.djebour@cyu.fr}.}}

\author[1]{\v{S}\'{a}rka Ne\v{c}asov\'{a} \thanks{E-mail Id: \href{mailto: matus@math.cas.cz}{matus@math.cas.cz} (corresponding author).}}
\affil[1]{Institute of Mathematics, Czech Academy of Sciences, \v{Z}itn\'{a} 25, 11000 Prague 1, Czechia}
\affil[2]{Department of Mathematics, CY Cergy Paris University, CNRS UMR 8088, 95302 Cergy-Pontoise, France}

\date{\today}
\pgfplotsset{width=7cm,compat=1.8}

\date{}

\begin{document}

\maketitle

\begin{abstract}
We investigate a fluid-structure interaction system in which the dynamics of the fluid  is described by the compressible Navier-Stokes equations, while the elastic structure is modeled by a damped plate equation. The fluid evolves in a 
three-dimensional bounded domain, with the structure occupying a part of its boundary. Instead of standard no-slip boundary conditions, we consider the Navier-slip boundary conditions  at the fluid-structure interface as well as  at the fixed boundary.  We establish the  local-in-time existence and uniqueness of strong solutions within  $L^{p}-L^{q}$ framework. The existence result is obtained for small time by decoupling the linearized system and employing a cascade strategy combined with the Second Schauder fixed point theorem, whereas the  uniqueness is shown by deriving weak regularity properties for the associated linear coupled operator in a Hilbert space setting.  It is the first result addressing strong solutions for a compressible fluid interacting with a damped plate under {\em Navier-slip} boundary conditions.
\end{abstract}

\vspace{1cm}
	
\noindent {\bf Keywords:} Fluid-structure interactions, compressible Navier-Stokes system, damped plate equation, strong solution, Navier-slip boundary conditions.
	
\vspace{.2cm}	
\noindent {\bf 2020 Mathematics Subject Classification.} 35Q30, 35R37, 76N10.

\tableofcontents
\vskip 10mm

\section{Introduction}

\subsection{Problem statement}
In this article, we consider a 
fluid-structure interaction problem in which the fluid flow is modeled by the compressible barotropic Navier-Stokes equations in a three-dimensional bounded domain, where part of the boundary deforms according to a damped plate equation. To be precise, 
let $\Omega_0 \subset \mathbb{R}^3$ be a bounded domain with smooth boundary that contains a flat portion denoted by
\[
\Gamma_{\rm flat} = \omega \times \{0\},
\]
where $\omega \subset \mathbb{R}^2$ is a smooth and bounded domain. The boundary of $\Omega_0$ is thus decomposed as
\[
\partial \Omega_0 = \Gamma_0 \cup \Gamma_{\rm flat}.
\]
The part $\Gamma_0$ is assumed to be rigid, while $\Gamma_{\rm flat}$ supports an elastic plate undergoing transversal deformations described by a scalar displacement $\eta$.
The corresponding deformed boundary is denoted by $\Gamma(\eta)$ and is defined as
\[
\Gamma(\eta)
=
\left\{
(x_1,x_2,x_3)\in \mathbb{R}^3
\;\middle|\;
(x_1,x_2)\in \omega,
\quad
x_3=\eta(x_1,x_2)
\right\}.
\]
We assume that a compressible fluid occupies the time-dependent domain $\Omega(\eta)$ enclosed by the rigid boundary $\Gamma_0$ and the moving interface $\Gamma(\eta)$.  In other words, the fluid domain $\Omega(\eta)$ is the interior of the closed simple and regular surface $\Gamma_0\cup \overline{\Gamma(\eta)}$. Moreover, we impose the no-contact condition
\[
\Gamma(\eta) \cap \Gamma_0 = \emptyset,
\]
ensuring that the elastic plate does not intersect the rigid part of the boundary.
	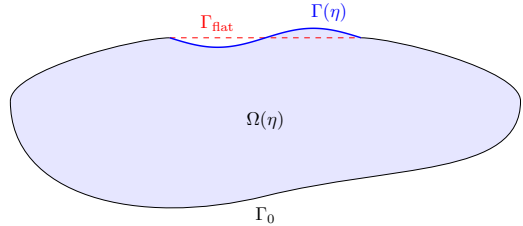
\begin{figure}[h] 
	\centering
	\scalebox{0.7}{%
\begin{tikzpicture}[scale=1.2]
	
	\fill[blue!10]
	(-4,0)
	.. controls (-4,-1.5) and (-2,-2) .. (0,-1.5)
	.. controls (2,-1) and (4,-1.2) .. (4,0)
	.. controls (4,0.5) and (2,1) .. (1.5,1)
	--
	plot[smooth, samples=100, domain=1.5:-1.5]
	(\x, {1 + 0.15*sin(180*\x/1.5)})
	--
	(-1.5,1)
	.. controls (-2,1) and (-4,0.5) .. (-4,0)
	-- cycle;
	
	\draw[thin]
	(-4,0)
	.. controls (-4,-1.5) and (-2,-2) .. (0,-1.5)
	.. controls (2,-1) and (4,-1.2) .. (4,0)
	.. controls (4,0.5) and (2,1) .. (1.5,1);
	
	\draw[thin]
	(-1.5,1)
	.. controls (-2,1) and (-4,0.5) .. (-4,0);
	
	\draw[red, dashed] (-1.5,1) -- (1.5,1);
	\node[red, above] at (-0.75,1) {$\Gamma_{\rm flat}$};
	
	\draw[blue, thick, smooth, samples=100, domain=-1.5:1.5]
	plot (\x, {1 + 0.15*sin(180*\x/1.5)});
	
	\node[blue, above] at (1,1.15) {$\Gamma(\eta)$};
	
	\node at (0,-0.3) {$\Omega(\eta)$};
	\node at (0,-1.8) {$\Gamma_0$};
	
\end{tikzpicture}

}
\caption{The fluid domain associated to the plate displacement $\eta$.}
\end{figure}
Then, the fluid-structure interaction system we want to study  is the following 
\begin{equation}
		\label{fs}
		\left\{
		\begin{array}{cc}
			\partial_t \widetilde{\rho} + \nabla\cdot(\widetilde{\rho} \widetilde{v})= 0 &      \text{ in } \, \Omega(\eta(t)),  \ t > 0 ,  \\
			\widetilde{\rho} (\partial_t\widetilde{v}+(\widetilde{v}\cdot \nabla)\widetilde{v})- \nabla\cdot(\sigma(\widetilde{v},\widetilde{p}))=0 &  \text{ in }  \,  \Omega(\eta(t)), \ t>0 ,\\
			 \partial_{tt}\eta + \alpha_1 \Delta^2 \eta - \alpha_2 \Delta\partial_t\eta= \mathbb{H}_\eta(\widetilde{v},\widetilde{p}) 
             & \text{ in } \,  \omega, \ t>0  . \\
		\end{array}
		\right.
\end{equation}
The function $\widetilde{\rho}$ represents the density of the fluid while $\widetilde{v}$ denotes the fluid velocity and $\widetilde{p}$ is the fluid pressure. The Cauchy stress tensor is defined by 
	\begin{equation*}
		\sigma(\widetilde{v},\widetilde{p}) = \mathbb{S}(\widetilde{v}) - \widetilde{p} \mathbb{I}_3,
	\end{equation*}
	with 
	\begin{equation}\label{Viscous-stress}
		\mathbb{S}(\widetilde{v}) = 2\mu \mathbb{D}(\widetilde{v}) + \lambda( \nabla\cdot \widetilde{v} )\mathbb{I}_3, \quad \mathbb{D}(\widetilde{v}) = \frac{1}{2}\left(\nabla \widetilde{v} + (\nabla \widetilde{v})^\top\right), 
\end{equation}
	where $\mu$ and $\lambda$ are the dynamic and the bulk viscosities of the fluid that verify
\begin{equation}
\label{1635}
\frac{2}{3}\mu+\lambda > 0,\quad \mu>0.
\end{equation}
	The pressure $\widetilde{p}=\widetilde{p}(\widetilde{\rho})$ that is given by the following constitutive law:
	$$
	\widetilde p(\widetilde{\rho})=\widetilde{\rho}^\gamma, \quad  \gamma \geq 1.
	$$ 
 The positive constants $\alpha_1$ and $\alpha_2$ in the structure equation correspond respectively to the rigidity and the damping on the structure. To simplify the presentation we consider $\alpha_1=\alpha_2=1$. The force of the fluid that acts on the moving interface is defined by
$$
\mathbb{H}_\eta(\widetilde{v},\widetilde{p})=- \sqrt{1+|\nabla\eta|^2}\sigma(\widetilde{v},\widetilde{p})n_\eta \cdot e_3, 
	$$
where we denote by  $n_\eta$ the unit outer normal vector on $\partial\Omega(\eta)$. In particular, 
$$
n_{\eta}= \frac{1}{\sqrt{1+|\nabla\eta|^2}}
\begin{pmatrix} 
	-\nabla\eta \\ 
	1 
\end{pmatrix} \ \text{ on } \ \Gamma(\eta) . $$  
Here $\nabla \eta = (\partial_{x_1}\eta, \partial_{x_2}\eta)^\top$ with the variables  $(x_1,x_2) \in \omega$.
We assume that the elastic plate is clamped along its boundary. More precisely, the displacement $\eta$ satisfies the clamped boundary conditions
\begin{equation}
	\label{clamped}
	\eta = 0, \ \ \nabla \eta \cdot n_\omega = 0
	\qquad \text{on } \partial \omega,\ t>0,
\end{equation}
where $n_\omega$ denotes the outward unit normal to $\partial \omega$. To introduce the boundary conditions, let us define the normal and the tangential components of any $z\in \mathbb{R}^3$ as follows
$$
z_{n_\eta}=(z\cdot n_\eta)\cdot n_\eta,\quad z_{\tau_\eta}=(\mathbb{I}_3 - n_\eta \otimes n_\eta)z.
$$
The fluid is subject to Navier-slip boundary conditions that read 
\begin{equation}
\label{bc}
\left\{
\begin{array}{cc}
\widetilde{v}_{n_{\eta}}=0,\quad \left[2\mu\mathbb{D}(\widetilde{v})n_\eta +\beta_0 \widetilde{v}\right]_{\tau_\eta} = 0 & \text{ on } \, \Gamma_0,  \ t>0 , \\
(\widetilde{v}-\partial_t \eta e_3 )_{n_{\eta}}=0,\quad \left[2\mu\mathbb{D}(\widetilde{v})n_\eta +\beta_1 (\widetilde{v}-\partial_t\eta e_3)\right]_{\tau_\eta} = 0 & \text{ on } \, \Gamma(\eta(t)), \ t>0 ,
\end{array}
\right. 
\end{equation}
where $\beta_0, \beta_1$ are the friction coefficients.
The Navier-slip boundary condition was proposed by Navier in 1823 \cite{Navier1823memoire} and is relevant in several physical
contexts, see for instance \cite{kistler1984coating,Liakos,Verfurth}.   In the context of fluid-structure interaction problems, 
Navier-slip boundary conditions have been considered in several recent works; see, for instance, \cite{Necasova-Ramas-Roy-Anja,Liu-Mitra-Necasova,Muha-Canic-JDE,zbMATH07093623}.  This Navier law states that, in one hand, 
the normal component of the velocity is continuous at the fluid-structure interface (impermeability condition) and that, on the other hand, the amount of slip in the tangential part
of the velocity is proportional to the tangential part of the normal stress exerted by the
fluid on the boundary. 
 For many years, the \emph{no-slip boundary condition}
has been the most widely used given its success in reproducing the
standard velocity profiles for incompressible/compressible viscous fluids.
Although the no-slip hypothesis seems to be in good agreement with
experiments, it leads to certain rather surprising conclusions. The
most striking one being the absence of collisions of rigid objects
immersed in a linearly viscous fluid \cite{HES,HIL}.
In the case of compressible fluid, the third author with her collaborators \cite{BNOR} proved that even in the case of no-slip boundary conditions, the  collision can occur exploiting the roughness of the body.
%
 Navier-slip boundary conditions provide a more realistic framework for describing fluid-structure interactions, especially near contact, and resolve the  “no-collision”  paradox encountered with classical no-slip conditions \cite{zbMATH03189127, zbMATH03278246, MR2481302}.  In fact, the authors in \cite{MR3281946}  studied the free fall of a sphere toward a wall in a viscous incompressible fluid, showing that while no-slip boundary conditions prevent finite-time contact, slip boundary conditions allow collision to occur as we mention before.

To complete the model, we consider  the following initial conditions:
\begin{equation}
	\label{ic}
	\left\{
	\begin{array}{c}
	\widetilde{\rho}(0,\cdot)=\widetilde{\rho}^0,\quad \widetilde{v}(0,\cdot)= \widetilde{v}^0 \  \ \ \  \ \ \text{ in } \Omega(\eta^{0}), \\
	\eta(0,\cdot)=\eta^0,\quad \partial_t \eta(0,\cdot)=\eta^1 \ \  \text{ in } \omega.
		\end{array}
	\right.
\end{equation}
Then, our goal  is to prove
the local-in-time existence of strong solution to  the  compressible fluid-structure interaction  system \eqref{fs}--\eqref{clamped} with  Navier-slip boundary conditions \eqref{bc} and initial data \eqref{ic}. 

\subsection{Bibliographic comments}

The  fluid-structure interaction problems have  been  analyzed  in various works.  Let us begin by  mentioning  the survey paper \cite{Quarteroni2000comp}, where the authors describe some models for blood flow in arteries as a fluid-structure interaction system. More specifically, they consider a simplified configuration in which the fluid is described by the incompressible Navier-Stokes equations and the structure is modeled by a damped beam equation. From the perspective of mathematical analysis, the existence of weak solutions to the corresponding model was investigated in \cite{Chambolle-Grandmont}, whereas the existence of strong solutions was proved in \cite{DaVeiga}.

The existence of weak solutions has also been established for more complex models. For instance, the interaction between   incompressible  Newtonian/ generalized Newtonian fluid and linear elastic Koiter shell were studied in \cite{Lengeler,Lengeler-Ruzicka}.  We also mention \cite{Muha-Canic}, where the authors consider dynamic pressure boundary conditions and model the structure either as a linear viscoelastic beam or as a linear elastic Koiter shell. These results were later extended in \cite{Muha-Canic-2,Muha-Canic-3} to a three-dimensional cylindrical domain, where the structure is described by an elastic cylindrical Koiter shell with not necessarily radially symmetric displacements.  We refer to the survey paper \cite{GLN} on existence and numerical results for a 2D/1D fluid-structure interaction model---an incompressible fluid (Newtonian and non-Newtonian) coupled with a thin elastic structure.

Let us  briefly review some results concerning the existence of strong solutions for incompressible fluid-structure interaction models. In \cite{Casanova}, the author studies a coupled system consisting of the incompressible Navier-Stokes equations and a damped Euler-Bernoulli beam equation, with boundary conditions involving the pressure. The work \cite{zbMATH07093623} addresses the interaction between  incompressible Navier-Stokes equations and  damped plate equation under Navier boundary conditions. In \cite{Debayan-Arnab-Raymond}, the authors establish the existence and uniqueness of maximal strong solutions for a model coupling the incompressible Navier-Stokes equations with a nonlinear shell structure. The paper \cite{Grandmont-Hillairet-ARMA} is devoted to the 
global-in-time existence and uniqueness of solutions for a damped beam model and, in particular, investigates the possibility of contact between the structure and the bottom of the domain.

We also mention \cite{Saal-Denk,Maity-Takahasi-JMFM}, where strong solutions are obtained within an $L^p-L^q$ 
 framework rather than the classical Hilbert space setting. Furthermore, in \cite{Lequeurre-JMFM}, the author proves the existence and uniqueness of strong solutions in the case of a damped wave equation. In \cite{Grandmont-Hillairet-Lequeurre}, three types of structures are considered: a damped beam equation, an undamped wave equation, and a beam equation with an inertia term. 
 Finally, the case of an undamped beam equation is studied in \cite{Badra-Takahasi-Funk-Ekvac,Badra-takahasi-SIMA-Gevrey} using the framework of Gevrey semigroups.


We now turn to the case of compressible fluids interacting with elastic structures through the boundary of the fluid domain. The global existence of weak solutions, up to the time when the structure comes into contact with the boundary, was established in \cite{zbMATH06860758, zbMATH01558334}. The local-in-time existence of strong solutions in the corresponding 2D/1D case  was obtained in \cite{zbMATH07273309}, and in \cite{zbMATH07339646}  when the structure is governed by a wave equation. We refer to \cite{Srdan-CPDE,Kaousek-JMPA,bhandari2025weaksolutionscompressiblemagnetohydrodynamic,Trifunovic-Wang} for results on the existence of weak solutions for compressible fluids interacting with elastic structures with no-slip condition. Finally, it is worth mentioning that the paper \cite{Breit-Malte-Sebastien} concerns with the existence of weak solutions for a compressible fluid interacting with a visco-elastic bulk solid, until the time of a collision.






In the $L^p-L^q$ framework, only a few results are available in the compressible setting. In \cite{zbMATH07391465}, the authors established local-in-time existence and global well-posedness for the compressible Navier-Stokes system in a general moving domain framework under 
no-slip boundary conditions. On the other hand, for a  fluid-beam interaction problem with no-slip boundary conditions, maximal regularity results were obtained in \cite{zbMATH07347904,KMNPW-L}. In that work, the authors employ a mixed change of variables combining Lagrangian and geometric transformations, together with a Banach fixed-point argument. The Lagrangian transformation plays a crucial role in handling the transport equation and, importantly, does not alter the boundary coordinates under no-slip boundary conditions. Last but not the least, we cite the work \cite{Hieber-Murata} where the authors addressed the local-in-time unique  strong solution for a  compressible fluid-rigid body interaction problem within the $L^p$ setting.  

\subsection{Main features  of the work}

In the present work, we consider 
Navier-slip boundary conditions to investigate the existence of strong solutions for a compressible 
fluid-structure interaction problems. To the best of our knowledge, this is the first result addressing strong solutions for a compressible fluid interacting with a damped plate under Navier-slip boundary conditions.

In contrast to the no-slip case, Navier-slip boundary conditions lead to additional difficulties, as the  powerful Lagrangian change of variables modifies the boundary geometry and thus perturbs the boundary conditions, even in the absence of structural coupling (see, e.g., \cite{zbMATH06298655}). To overcome this issue, we employ a {\em geometric change of variables} introduced in \cite{zbMATH06693017}.  This transformation allows us to handle the transport equation while preserving the normal continuity condition on the entire boundary. Moreover, since the transformation coincides with the identity on the fixed part $\Gamma_{0}$, the slip condition is preserved there, and the nonlinear effects arise only on the moving boundary.   

\vspace*{0.2cm}

The main difficulties and novelties of this work can be summarized as follows:
\begin{itemize}
	\item As mentioned above, the Lagrangian change of variables, commonly denoted by $X$, is the most natural and effective choice in the compressible setting. However, in  presence of Navier-slip boundary conditions—unlike the no-slip case—no explicit representation of the boundary coordinates $X(y)$ is available. {\em This lack of boundary control constitutes a major obstruction to the direct application of the Banach fixed-point theorem.} 
    
	\item To address this difficulty, we use a {\em geometric change of variables}, originally introduced in \cite{zbMATH06693017} and described in detail in \cref{chV} via \eqref{lag}, along with the corresponding change of unknowns \eqref{1807}. 
    This transformation preserves the impermeability boundary condition and allows for an effective treatment of the transport equation. Although it is classically used in incompressible settings, it turns out to be particularly well suited to our compressible framework.

	\item This change of variables impacts the entire coupled system, including the slip boundary conditions. The resulting nonlinear terms in the boundary conditions are treated in the negative Sobolev space $\mathscr W^{-1,q}$, defined  in \cref{not}  by \eqref{1213}. This functional setting is essential in order to avoid the use of nonlocal operators acting on the time derivative; see for instance \cite{zbMATH06298655,zbMATH06539947}.

	\item To establish the existence of strong solutions, we employ the {\em Second Schauder fixed-point theorem} (see Theorem \ref{Thm-Fixed-Point} of the present paper). However, a limitation of this theorem is that it does not guarantee uniqueness of the solution. To address this issue, we work with a weaker norm in a Hilbert space setting and derive parabolic regularity estimates for a perturbed and coupled fluid-structure operator (described in Lemma \ref{1821}). We got the inspiration of using the {\em weaker norm} for proving the uniqueness from the work \cite{zbMATH05674979}, where the author considered the  compressible Navier-Stokes equations (without structural interaction) under homogeneous Dirichlet boundary conditions and in such situation, he needs the weaker norm in $L^p-L^q$ setting.
  In our case, we need to deal with structure on the boundary (producing a moving domain framework) which gives additional difficulties in the whole system.

    
\end{itemize}

\subsection{Notations}
\label{not}

Let $d\in \mathbb{N}^*$. For any Banach space $E$, its norm will be denoted by $\|\cdot\|_{E}$, while $E^d$ stands for the $d$-product of $E$ and its norm will be denoted by $\|\cdot\|_{E^d}$. For any domain $\Omega$, $L^p(\Omega)$ and $W^{m,p}(\Omega)$ are the classical Lebesgue and Sobolev spaces with $1\leq p\leq \infty$, $m\in \mathbb{N}$. For any $T>0$, we define also the spaces of type 
$L^p(0,T;E^d)$, $W^{m,p}(0,T;E^d)$,
as the usual Lebesgue and Sobolev spaces with 
$E^d$-valued functions defined on $(0,T)$. In the sequel, their norms will be denoted respectively by $\|\cdot\|_{L^p(0,T;E^d)}$, $\|\cdot\|_{W^{m,p}(0,T;E^d)}$. For $p\geq 1$, we denote by $p'$ its dual exponent. If $p=2$, we will use the classical notation $H^m=W^{m,2}$. Let $\ell,m\in\mathbb{N}$ such that $\ell<m$.  For any $1\leq p,q<\infty$, we recall the classical definition of the Besov spaces by the interpolation 
$$
B^s_{qp}(\Omega)=(W^{\ell,q}(\Omega),W^{m,q}(\Omega))_{\theta p},\quad s=\theta \ell+(1-\theta)m,\quad \theta\in(0,1). 
$$ 
For more details about Besov spaces, we refer to \cite{zbMATH05779185}. Let us also recall the Fourier transform of a $L^1(\mathbb{R}^3)$ function and its inverse
$$
\mathcal{F}[u](\xi)=\int_{\mathbb{R}^3}e^{-ix\xi} u(x)\ dx,\quad \mathcal{F}^{-1}[u](x)=\frac{1}{(2\pi)^3}\int_{\mathbb{R}^3}e^{ix\xi} u(\xi)\ d\xi.
$$
We define the operator $(\mathbb{I}-\Delta)^{-1/2}$ as follows
$$
 (\mathbb{I}-\Delta)^{-1/2} u=\mathcal{F}^{-1}[(1+|\xi|^2)^{-1/2}\mathcal{F}[u](\xi)].
$$
Let $\Omega \subset \mathbb R^3$ and $\imath: L^1_{loc}(\Omega)\rightarrow L^1_{loc}(\mathbb{R}^3)$ be the extension operator that satisfies the following properties for any $q\in (1,\infty)$ (see, e.g., \cite{zbMATH06496036}),
$$
\forall u\in W^{1,q}(\Omega),\quad  \imath u=u \text{ in }\Omega,
$$
$$
\forall u\in W^{1,q}(\Omega),\quad \|\imath u\|_{W^{k	,q}(\mathbb{R}^3)}\leq C\|u\|_{W^{k,q}(\Omega)},\quad k=0,1,
$$
and 
$$
\forall u\in W^{1,q}(\Omega),\quad  \|(\mathbb{I}-\Delta)^{-1/2}\imath \nabla u\|_{L^q(\mathbb{R}^3)}\leq C \|u\|_{L^q(\Omega)},\quad k=0,1.
$$
We further  define $\mathscr{W}^{-1,q}(\Omega)$ as 
\begin{equation}
\label{1213}
\mathscr{W}^{-1,q}(\Omega)=\{ u\in L^1_{loc}(\Omega)\; |\; (\mathbb{I}-\Delta)^{-1/2}\imath u\in L^q(\Omega)\},
\end{equation}
endowed with the norm
$$
\|u\|_{\mathscr{W}^{-1,q}(\Omega)}=\|(\mathbb{I}-\Delta)^{-1/2}\imath u\|_{L^q(\Omega)}. 
$$


\subsection{Main result} 
Let us now state the main result of our work. 
\begin{Theorem}[Existence and uniqueness]
\label{Thm-main}
Let $p>2$, $q>3$ with $\frac{1}{p}+\frac{1}{2q}<\frac{1}{2}$ and suppose that the  initial data satisfy:
\begin{equation}
\label{initialD-Main}
\begin{array}{c}
\widetilde \rho^0\in W^{1,q}(\Omega(\eta^0)),\quad \min_{\overline{\Omega(\eta^0)}}\widetilde \rho^0>0,\\ \vspace*{0,1cm}
\eta^0\in B_{qp}^{2(2-1/p)}(\omega),\quad \min_{\overline{\omega}}\left( \eta^0+1\right) >0,\\ \vspace*{0,1cm}
\widetilde v^0\in [B_{qp}^{2(1-1/p)}(\Omega(\eta^0))]^3, \quad \eta^1\in B_{qp}^{2(1-1/p)}(\omega),\\ \vspace*{0,1cm}
\widetilde v^0_{n_{\eta^0}}=0 \ \text{ on } \  \Gamma_{0},\quad (\widetilde v^0-\eta^1e_3)_{n_{\eta^0}}=0 \ \text{ on } \ \Gamma(\eta^0), \\
 \eta^0 = \nabla \eta^0 \cdot n_\omega =0 \ \text{ on } \ \partial \omega . 
\end{array}
\end{equation}
Then there exists  $T>0$ such that the system \eqref{fs}--\eqref{ic} admits a unique solution in the time interval $(0,T)$ with
$$
\widetilde \rho\in L^\infty(0,T;W^{1,q}(\Omega(\eta)))\cap W^{1,p}(0,T;L^q(\Omega(\eta))),
$$
$$
\widetilde v\in L^p(0,T;[W^{2,q}(\Omega(\eta))]^3)\cap C([0,T]; [B_{pq}^{2(1-1/p)}(\Omega(\eta))]^3)\cap W^{1,p}(0,T;[L^q(\Omega(\eta))]^3),
$$
$$
\eta\in L^p(0,T;W^{4,q}(\omega))\cap C([0,T];B_{qp}^{2(2-1/p)}(\omega))\cap W^{2,p}(0,T;L^q(\omega)),
$$
$$
\partial_t\eta\in L^p(0,T;W^{2,q}(\omega))\cap C([0,T];B_{qp}^{2(1-1/p)}(\omega))\cap W^{1,p}(0,T;L^q(\omega)).
$$
\end{Theorem}

\vspace{.2cm}
\noindent 
{\bf Paper organization.}
Rest of the paper is organized as follows. In \cref{chV}, we introduce a suitable change of variables that maps the problem onto a fixed cylindrical domain. This reformulation generates nonlinear perturbations throughout the system, including within the Navier-slip boundary conditions. In \cref{lin}, we investigate the resulting linearized system by means of a cascade approach. Finally, in \cref{fp}, we prove the existence of strong solutions   by applying the so-called {\em Second Schauder fixed-point theorem}  for sufficiently small time $T$. The  uniqueness of solution is then established by deriving weak regularity estimates for an auxiliary linear coupled fluid-structure system, see subsection \ref{Sec-Uniquness} for details. Several auxiliary results are addressed in Appendices \ref{Apndx-A}, \ref{Apndx-Calcul}, \ref{Apndx-Weak-regu}.

\section{Change of variables}
\label{chV}
    
This change of variables is inspired from \cite{zbMATH06693017}.  
Let us define a mapping $X(t,\cdot)$ that transforms the reference domain
$\Omega:=\Omega(\eta^0)$ into the time-dependent domain $\Omega(\eta(t))$.
Let $\varepsilon>0$ and consider the open set
$$
\mathcal{O}_{\eta^0,\varepsilon}
=
\left\{
(y_1,y_2,y_3)\in \mathbb{R}^3
\;\middle|\;
(y_1,y_2)\in \omega,
\quad
\eta^0(y_1,y_2)-\varepsilon < y_3 < \eta^0(y_1,y_2)
\right\}.
$$
We choose $\varepsilon$ sufficiently small so that
$\mathcal{O}_{\eta^0,\varepsilon}\subset \Omega$.
Since $\eta^0\in W^{2,q}(\omega)$ with $q>3$, it admits an extension by zero
outside $\omega$ belonging to $W^{2,q}(\mathbb{R}^2)$.
Let $\alpha\in C_0^\infty(\mathbb{R})$ be a cut-off function such that
$\alpha\equiv 1$ on $(-\varepsilon/2,\varepsilon/2)$ and
$\alpha\equiv 0$ on $\mathbb{R}\setminus(-\varepsilon,\varepsilon)$.
We define
$$
\theta(y) := \alpha\bigl(y_3-\eta^0(y_1,y_2)\bigr).
$$
Observe that
\begin{equation}
	\label{theta}
	\theta(y)=
	\begin{cases}
		1, & y\in \mathcal{O}_{\eta^0,\varepsilon/2},\\[0.3em]
		0, & y\in \mathbb{R}^3\setminus \mathcal{O}_{\eta^0,\varepsilon}.
	\end{cases}
\end{equation}

We now introduce the following change of variables:
\begin{equation}
	\label{lag}
	X(y)=
	\begin{pmatrix}
		y_1 \\[0.2em]
		y_2 \\[0.2em]
		y_3+\theta(y)\bigl(\eta(y_1,y_2)-\eta^0(y_1,y_2)\bigr)
	\end{pmatrix},
	\qquad y\in \mathbb{R}^3,\; t>0.
\end{equation}
By construction, one has $X(\Omega)=\Omega(\eta(t))$. Moreover, provided that
$\eta$ remains sufficiently close to $\eta^0$, the mapping $X$ defines a
$C^1$-diffeomorphism from $\Omega$ onto $\Omega(\eta(t))$.
Indeed, $X\in C^1(\mathbb{R}^3)$ since $\eta-\eta^0\in W^{2,q}(\omega)$ and thus
admits an extension in $W^{2,q}(\mathbb{R}^2)\hookrightarrow C^1(\mathbb{R}^2)$.   We further introduce the following:
\begin{align*}
&\delta_X = 1 + \partial_{y_3} \theta (\eta - \eta^0) , \quad \text{ and} \\
&a_i = (\eta-\eta^0)\partial_{y_i}\theta + \theta \partial_{y_i} (\eta-\eta^0) , \ \ i=1,2,
\end{align*}
and define
\begin{align*}
    \nabla X = \begin{pmatrix}
    1 &  0 & 0 \\
    0 & 1 & 0 \\
    a_1 & a_2 & \delta_X \end{pmatrix},
    \ \  b_X : = \text{Cof }\nabla X = \begin{pmatrix}
\delta_X & 0 & -a_1\\
0 & \delta_X & -a_2\\
0 & 0 & 1
\end{pmatrix}, \ \ 
\partial_t X = \theta\partial_t\eta e_3.
\end{align*}
In particular, the Jacobian determinant of $X$ satisfies
$$
\det\nabla X =\delta_X= 1 + \alpha'\bigl(y_3-\eta^0(y_1,y_2)\bigr)\bigl(\eta-\eta^0\bigr).
$$

Consequently, $X$ is a $C^1$-diffeomorphism of $\mathbb{R}^3$ provided that
$$
\|\theta'\|_{L^\infty(\mathbb{R})}
\|\eta-\eta^0\|_{L^\infty(\omega)} < 1.
$$
This condition is ensured for sufficiently small time $T$, as will be made
precise in \cref{fp}. Finally, note that $X\equiv \mathbb{I}_3$ in
$\Omega\setminus \mathcal{O}_{\eta^0,\varepsilon}$ and that
$X(\partial\Omega)=\partial\Omega(\eta)$. Hence, $X$ is a $C^1$-diffeomorphism
from $\Omega$ onto $\Omega(\eta)$, and we denote its inverse by $Y$.

Now  consider  
\begin{equation}
\label{1807}
v(t,y)=b_X^\top\widetilde{v}(t,X(t,y)),\quad \rho(t,y)=\widetilde{\rho}(t,X(t,y)). 
\end{equation}
	We have the following relation
	$$
	N_{\eta}=b_X N_{\eta^0},
	$$
	where $N_{\eta}$ (resp. $N_{\eta^0}$) designates the outward normal vector on $\partial\Omega(\eta(t))$ (resp. $\partial\Omega$) defined by
	$$
	N_\eta=\begin{pmatrix}
		-\nabla\eta \\ 1
	\end{pmatrix},\quad  N_{\eta^0}=	\begin{pmatrix}
		-\nabla\eta^0 \\ 1
	\end{pmatrix}.
	$$
     In fact, one has $N_\eta=\sqrt{1+|\nabla \eta|^2} \, n_\eta$.
	Consequently, 
	$$
	\tau_{\eta}^j=({b_X^{-1}})^\top (\tau_{\eta^0})^j,\quad j=1,2,
	$$
	where $(\tau^j_{\eta})_{j=1,2}$ (resp. $(\tau^j_{\eta^0})_{j=1,2}$) denotes  a basis of the tangent subspace on each point of $\partial\Omega(\eta(t))$ (resp. $\partial\Omega$). 
	To simplify the notations, we will write in the sequel
	$$
	n_{\eta^0}=:n,\quad N_{\eta^0}=:N,\quad \tau_{\eta^0}^j=:\tau^j,\ j=1,2. 
	$$
	The equation\eqref{fs}$_1$ gives
	$$
	\delta_X\partial_t\rho+(v-b_X^\top\partial_tX)\nabla \rho=-\rho \nabla\cdot v.
	$$
	In this case, $b_X^\top\partial_tX=\partial_t X$. Hence, the transport equation writes
	\begin{equation}
		\label{dens}
		\delta_X\partial_t\rho+(v-\partial_tX)\nabla \rho=-\rho \nabla\cdot v.
	\end{equation}
	For the remaining terms, we can refer to \cite{DJEBOUR2024104022} for similar calculations. Let us denote 
	\begin{equation}
		\label{defT}
		[\mathbb{T}_X(\rho,v,\eta)]_i: =\rho \left( (\nabla X) \partial_t v+(\nabla X)(\nabla v)\partial_t Y(X)+\delta_X\partial_t (b_X^{-\top})(X) v\right) ,\quad i=1,2,3,
	\end{equation}
    \begin{multline}
        \label{defD}
        \left[\mathbb{D}_X (v)\right]_{ij}
		: =\frac{1}{2}\bigg(
		\sum_{mk}\left(\frac{\partial (b_X^{-\top})_{ik}}{\partial x_m}(X)+\frac{\partial (b_X^{-\top})_{mk}}{\partial x_i}(X)\right) v_k (b_X)_{mj}
		\\+
		\sum_{m,k,l}\left( 
		(b_X^{-\top})_{ik}(X)\frac{\partial Y_l}{\partial x_m}(X)
		+(b_X^{-\top})_{mk}(X)\frac{\partial Y_l}{\partial x_i}(X)
		\right) \frac{\partial v_k}{\partial y_l}
		(b_X)_{mj}\bigg),
    \end{multline}
	\begin{equation}
		\label{defS}
		 \left[\mathbb{S}_X (v)\right]_{ij}= 2\mu\left[\mathbb{D}_X (v)\right]_{ij}+\lambda\frac{\nabla\cdot v}{\delta_X}(b_X)_{ij},\quad i,j=1,2,3,
	\end{equation} 
	and
	\begin{multline}\label{defN}
		[\mathbb{N}_X(\rho,v,\eta)]_i:=\rho\left(\nabla\cdot \left( (b^{-\top}_{X}v\otimes b^{-\top}_{X} v)b_{X}\right)-(\nabla\cdot v )b^{-\top}_{X}v\right)\\=\rho\left( 
		\sum_{k,j,m} \frac{\partial (b_X^{-\top})_{ik}}{\partial x_j}(X) \frac{\partial X_j}{\partial y_m} v_k v_m
		+\sum_{k,l}(b_X^{-\top})_{ik}(X)\frac{\partial v_k}{\partial y_l} v_l\right) ,\quad i = 1,2,3.
	\end{multline}
	After standard computations, we get
	\begin{equation}\label{15:39}
		\delta_X\widetilde{\rho}(X)\partial_t \widetilde{v}(X) = \mathbb{T}_X(\rho,v,\eta),
	\end{equation}
	\begin{equation}\label{15:40}
		\delta_X \nabla\cdot \left( \mathbb{S} (\widetilde{v}) \right)(X)=\nabla\cdot \left(\mathbb{S}_X (v)\right), 
		\ \ 
		\delta_X \nabla \widetilde{p}(X)=\gamma \rho^{\gamma-1}b_X \nabla \rho,
		\ \  \delta_X \widetilde{\rho}(X)\left[(\widetilde{v}\cdot \nabla)\widetilde{v}\right](X)=\mathbb{N}_X(\rho,v,\eta),
	\end{equation}
    where $\mathbb S$ is the viscous stress tensor defined by \eqref{Viscous-stress}. 
    
    Thus, the system \eqref{fs}--\eqref{ic} is rewritten in a cylindrical domain as 
	\begin{equation}\label{fs1}
		\left\{
		\begin{array}{cc}
			\partial_t\rho+\frac{1}{\delta_X}(v-\partial_tX)\nabla \rho  = -\frac{\rho}{\delta_X} \nabla \cdot v 
         &\quad  \text {in } (0,T) \times \Omega, \vspace{0.15cm} \\
			\rho^0\partial_t v - \nabla\cdot \left(\mathbb{S}(v)\right) = \mathbb{F}_X(\rho,v,\eta) &\quad  \text {in } (0,T) \times \Omega, \vspace{0.15cm} \\
			\partial_{tt}\eta + \Delta^2 \eta -\Delta \partial_t\eta= \mathbb{G}_X(\rho,v,\eta)& \text{in } (0,T)\times\omega,\\
		\end{array}
		\right.
	\end{equation}
	with the boundary conditions 
	\begin{equation}
		\label{bc1}
		\left\{
		\begin{array}{cc}
			v_n = 0,\quad \quad \left[2\mu\mathbb{D}(v)n +\beta_0 v\right] \cdot \tau^j=0 &\text{on } (0,T)\times\Gamma_{0},\quad j=1,2, \vspace{0.15cm}\\
			\left(v-\partial_t \eta e_3 \right)_{n}=0,\quad \left[2\mu\mathbb{D}(v)n +\beta_1 (v-\partial_t\eta e_3)\right] \cdot \tau^j = [\mathbb{K}_X(\rho,v,\eta)]_j&\text{on } (0,T)\times\Gamma (\eta^0),\quad j=1,2,
		\end{array}
		\right.
	\end{equation}
	and the initial conditions
	\begin{equation}
		\label{ic1}
		\rho(0,\cdot)= \rho^0, \ \ v(0,\cdot)= v^0 \ \ \text{in } \, \Omega,   \quad
		\eta(0,\cdot)=\eta^0, \ \ \partial_t \eta(0,\cdot)=\eta^1 \ \ \text{in } \, \omega,
	\end{equation}
	\begin{equation}
		\label{F}
	\mathbb{F}_X(\rho,v,\eta)=\left( \rho^0\partial_t v-\mathbb{T}_X(\rho,v,\eta)\right) +\nabla \cdot (\mathbb{S}_X(v)-\mathbb{S}(v))-\mathbb{N}_X( \rho,v,\eta)-\gamma \rho^{\gamma-1}b_X \nabla \rho,
	\end{equation}
	\begin{equation}
		\label{G}
		\mathbb{G}_X(\rho,v,\eta)=-\mathbb{S}_X(v)N \cdot e_3+\rho^\gamma,
	\end{equation}
	\begin{equation}
		\label{K}
	[	\mathbb{K}_X(\rho, v,\eta)]_j=\left[2\mu \left( \mathbb{D}(v)-b_X^{-1}\frac{|N|}{|N_\eta|}\mathbb{D}_X(v) \right)n +\beta_1\left(\mathbb{I}_3-b_X^{-1}b_X^{-\top} \right)v+\beta_1\left(b_X^{-1}-\mathbb{I}_3\right)\partial_t\eta e_3 \right] \cdot \tau^j, \ \  j=1,2.
	\end{equation}

After applying the above change of variables, Theorem \ref{Thm-main} can be reformulated as follows. The corresponding initial data, defined in \eqref{initialD-Main}, now satisfy
\begin{equation}
\label{initialD}
\begin{array}{c}
 \rho^0\in W^{1,q}(\Omega),\quad \min_{\overline{\Omega}}\rho^0>0,\\ \vspace*{0,1cm}
\eta^0\in B_{qp}^{2(2-1/p)}(\omega),\quad \min_{\overline{\omega}}\left( \eta^0+1\right) >0,\\ \vspace*{0,1cm}
v^0\in [B_{qp}^{2(1-1/p)}(\Omega)]^3, \quad \eta^1\in B_{qp}^{2(1-1/p)}(\omega),\\ \vspace*{0,1cm}
v^0_n=0 \ \text{ on } \ \Gamma_{0},\quad (v^0-\eta^1e_3)_n=0 \ \text{ on } \ \Gamma(\eta^0), \\
\eta^0 = \nabla \eta^0 \cdot n_\omega = 0 \ \text{ on } \partial \omega ,
\end{array}
\end{equation}
where $n$ is the unit outward normal vector to $\partial\Omega$. 
\begin{Theorem}[Existence and uniqueness---reformulated]
\label{main}
Let $p>2$, $q>3$ with $\frac{1}{p}+\frac{1}{2q}<\frac{1}{2}$ and suppose that the initial data satisfy \eqref{initialD}.
Then there exists  $T>0$ such that the system \eqref{fs1}--\eqref{ic1} admits a unique strong solution in the time interval  $(0,T)$ with
$$
\rho\in L^\infty(0,T;W^{1,q}(\Omega))\cap W^{1,p}(0,T;L^q(\Omega)),
$$
$$
v\in L^p(0,T;[W^{2,q}(\Omega)]^3)\cap C([0,T]; [B_{pq}^{2(1-1/p)}(\Omega)]^3)\cap W^{1,p}(0,T;[L^q(\Omega)]^3),
$$
$$
\eta\in L^p(0,T;W^{4,q}(\omega))\cap C([0,T];B_{qp}^{2(2-1/p)}(\omega))\cap W^{2,p}(0,T;L^q(\omega)),
$$
$$
\partial_t\eta\in L^p(0,T;W^{2,q}(\omega))\cap C([0,T];B_{qp}^{2(1-1/p)}(\omega))\cap W^{1,p}(0,T;L^q(\omega)).
$$
\end{Theorem}

Some remarks are in order. 
\begin{Remark}
\label{1051}
The condition $1/p+1/2q<1/2$ ensures that the trace mapping from $B_{qp}^{1-2/p}(\Omega)$ into $L^q(\partial\Omega)$ is continuous, see for instance \cite[p. 200]{zbMATH05779185}. This is important to manage some terms in the fixed point part.  
\end{Remark}
\begin{Remark}
\label{spaces}
For notational convenience, each unknown will be identified with its corresponding function space
$$
\mathcal{X}_\rho=\{\rho\in C([0,T];W^{1,q}(\Omega))\cap W^{1,p}(0,T;L^q(\Omega))\},
$$
$$
\mathcal{X}_v=\{v\in L^p(0,T;[W^{2,q}(\Omega)]^3)\cap W^{1,p}(0,T; [L^q(\Omega)]^3)\},
$$
$$
\mathcal{X}_\eta=\{\eta \in L^p(0,T;W^{4,q}(\omega))\cap W^{2,p}(0,T; L^q(\omega)), \quad \partial_t\eta\in L^p(0,T;W^{2,q}(\omega))\cap W^{1,p}(0,T; L^q(\omega))\}.
$$
On the other hand, we have the following embeddings (see, for instance \cite[Theorem 4.10.2, p. 180]{zbMATH00742737})
$$
L^p(0,T;[W^{2,q}(\Omega)]^3)\cap W^{1,p}(0,T; [L^q(\Omega)]^3) \hookrightarrow C([0,T]; [B_{pq}^{2(1-1/p)}(\Omega)]^3),
$$
and 
$$
L^p(0,T;W^{4,q}(\omega))\cap W^{2,p}(0,T; L^q(\omega)) \hookrightarrow C([0,T];B_{qp}^{2(2-1/p)}(\omega)),
$$
$$
L^p(0,T;W^{2,q}(\omega))\cap W^{1,p}(0,T; L^q(\omega)) \hookrightarrow C([0,T];B_{qp}^{2(1-1/p)}(\omega)). 
$$
Then, we set their corresponding norms
$$
\|\rho\|_{\mathcal{X}_\rho}=\|\rho\|_{C([0,T];W^{1,q}(\Omega)) }+ \|\rho\|_{W^{1,p}(0,T;L^{q}(\Omega))} , 
$$
$$
\|v\|_{\mathcal{X}_v}=\|v\|_{L^p(0,T;[W^{2,q}(\Omega)]^3)}+ \|v\|_{W^{1,p}(0,T; [L^q(\Omega)]^3)}+ \|v\|_{C([0,T]; [B_{pq}^{2(1-1/p)}(\Omega)]^3)},
$$
\begin{multline*}
\|\eta\|_{\mathcal{X}_\eta}=\|\eta\|_{L^p(0,T;W^{4,q}(\omega))}+\|\eta\|_{W^{2,p}(0,T; L^q(\omega))} + \|\eta\|_{C([0,T];B_{qp}^{2(2-1/p)}(\omega))}  \\
+ \|\partial_t\eta\|_{L^p(0,T;W^{2,q}(\omega))}+\|\partial_t\eta\|_{W^{1,p}(0,T; L^q(\omega))}+ \|\partial_t\eta\|_{C([0,T];B_{qp}^{2(1-1/p)}(\omega))}.
\end{multline*}
Since $p>2$, we have the following embeddings
$$
B_{qp}^{2(1-1/p)}(\Omega) \hookrightarrow W^{1,q}(\Omega),\quad B_{qp}^{2(2-1/p)}(\omega) \hookrightarrow W^{3,q}(\omega).
$$
In particular, this implies: 
$v\in C([0,T]; [W^{1,q}(\Omega)]^3)$, $\eta \in C([0,T]; W^{3,q}(\omega))$ and  $ \partial_t\eta \in C([0,T]; W^{1,q}(\omega))$.
\end{Remark}

\section{Linearized system}
	\label{lin}
	 We consider a set of functions $(\widehat{v},\widehat{\eta})$  such that
	$$
	\widehat{v}_n=0\quad \text{ on }\Gamma_{0},\quad (\widehat{v}-\partial_t\widehat{\eta}e_3)_n=0\quad \text{ on }\Gamma(\eta^0).
	$$
Let $\widehat{X}(t,\cdot)=X_{\widehat{\eta}(t)}(\cdot)$. For this given $(\widehat v, \widehat \eta)$, first we consider   the unique solution $\widehat \rho$ to the linearized transport equation 
\begin{align}\label{linearized-transport-rho-hat}
\left\{
		\begin{array}{cc}
\partial_t\widehat{\rho}+\frac{1}{\delta_{\widehat{X}}}(\widehat{v}-\partial_t\widehat{X})\nabla \widehat{\rho} =-\frac{\widehat \rho}{\delta_{\widehat X}} \nabla\cdot \widehat v &\quad  \text {in } (0,T) \times \Omega, \\
\widehat \rho(0)=\rho^0 & 
\quad \text{in } \Omega.
 \end{array}
		\right.
    \end{align}
With this $\widehat \rho$ and $(\widehat v, \widehat \eta)$,
we consider the linearized system around $(\widehat{v},\widehat{\eta})$ as follows:
	\begin{equation}\label{fs1*}
		\left\{
		\begin{array}{cc}
			\rho^0\partial_t v - \nabla\cdot \left(\mathbb{S}(v)\right) = \mathbb{F}_{\widehat{X}}(\widehat{\rho},\widehat{v},\widehat{\eta})&\quad  \text {in } (0,T) \times \Omega,  \vspace{0.15cm}\\
			\partial_{tt}\eta + \Delta^2 \eta -\Delta \partial_t\eta= \mathbb{G}_{\widehat{X}}(\widehat{\rho},\widehat{v},\widehat{\eta})  & \text{in } (0,T)\times\omega,\\
           v(0)=v^0   &\text{in } \Omega ,   \\
             \eta(0)=\eta^0, \ \ \partial_t \eta(0)=\eta^1 & \text{in } \omega ,  
		\end{array}
		\right.
	\end{equation}
	with  boundary conditions 
	\begin{equation}
		\label{bc1*}
		\left\{
		\begin{array}{cc}
			v_n = 0,\quad \quad \left[2\mu\mathbb{D}(v)n +\beta_0 v\right] \cdot \tau^j=0 &\text{ on } (0,T)\times\Gamma_{0},  \vspace{0.15cm}\\
			\left(v-\partial_t \eta e_3 \right)_{n}=0,\quad \left[2\mu\mathbb{D}(v)n +\beta_1 (v-\partial_t{\eta} e_3)\right] \cdot \tau^j = [\mathbb{K}_{\widehat{X}}(\widehat{\rho},\widehat{v},\widehat{\eta})]_j&\text{ on } (0,T)\times\Gamma (\eta^0), 
		\end{array}
		\right.
	\end{equation}
    for $j=1,2$.

    Basically, we analyze the following linear system for given functions  $f$, $g$,  and $k_j$, 
	\begin{equation}\label{fs1l}
		\left\{
		\begin{array}{cc}
		\rho^0\partial_t v - \nabla\cdot \left(\mathbb{S}(v)\right) = f&\quad  \text {in } (0,T) \times \Omega,  \vspace{0.15cm}\\
			\partial_{tt}\eta + \Delta^2 \eta - \Delta \partial_t\eta= g  & \text{in } (0,T)\times\omega,\\\end{array}
		\right.
	\end{equation}
	with the boundary conditions 
	\begin{equation}
		\label{bc1l}
		\left\{
		\begin{array}{cc}
				v_n = 0,\quad \quad \left[2\mu\mathbb{D}(v)n +\beta_0 v\right] \cdot \tau^j=0 & \text{ on } (0,T)\times\Gamma_{0},  \vspace{0.15cm}\\
			\left(v-\partial_t \eta e_3 \right)_{n}=0,\quad \left[2\mu\mathbb{D}(v)n +\beta_1 (v-\partial_t\eta e_3)\right] \cdot \tau^j = k_j & \text{ on } (0,T)\times\Gamma (\eta^0) ,
		\end{array}
		\right.
	\end{equation}
for $j=1,2$.

\subsection{The plate equation}
Here, we  consider the damped beam equation
\begin{equation}
\left\{
\begin{array}{cc}
\label{plate}
\partial_{tt}\eta + \Delta^2 \eta -\Delta\partial_t\eta = g & \text{in } (0,T)\times\omega,  \vspace{0.15cm}\\
\eta=\nabla \eta\cdot n_\omega=0 & \text{ on }(0,T)\times\partial\omega,  \vspace{0.15cm}\\
\eta(0)=\eta^0, \ \ \partial_t\eta(0)=\eta^1   & \text{in } \omega . \\
\end{array}\right.
\end{equation}
The following result was established in \cite{zbMATH06436265}.

\begin{Theorem}\label{Theorem-eta-sol}
Let 
$g\in L^p(0,T;L^q(\omega))$, $\eta^0\in B_{qp}^{2(2-1/p)}(\omega)$ and $\eta^1\in B_{qp}^{2(1-1/p)}(\omega)$  satisfying the boundary conditions $\eqref{initialD}_5$.  
Then \eqref{plate} admits a unique solution $\eta\in\mathcal{X}_\eta$ satisfying the estimate
\begin{equation}
\label{1118}
\|\eta\|_{\mathcal{X}_\eta}\leq C\left(\|\eta^0\|_{B_{qp}^{2(2-1/p)}(\omega)}+\|\eta^1\|_{B_{qp}^{2(1-1/p)}(\omega)}+\|g\|_{L^p(0,T;L^q(\Omega))}\right). 
\end{equation}
\end{Theorem}
\begin{Remark}
We recall  the embeddings $B_{qp}^{2(2-1/p)}(\omega)\hookrightarrow W^{3,q}(\omega)$ and $B_{qp}^{2(1-1/p)}(\omega)\hookrightarrow W^{1,q}(\omega)$. Therefore, 
$\eta\in C([0,T];W^{3,q}(\omega))$, $\partial_t\eta\in C([0,T];W^{1,q}(\omega))$. In addition, due to \eqref{1118} we have 
$$
\|\eta\|_{C([0,T];W^{3,q}(\omega))}+\|\partial_t\eta\|_{C([0,T];W^{1,q}(\omega))}\leq C\left(\|\eta^0\|_{B_{qp}^{2(2-1/p)}(\omega)}+\|\eta^1\|_{B_{qp}^{2(1-1/p)}(\omega)}+\|g\|_{L^p(0,T;L^q(\Omega))}\right) .   
$$
\end{Remark}

\subsection{The Lamé system}

Let us now study the Lamé system with 
nonhomogeneous boundary condition
\begin{equation}\label{heat}
		\left\{
		\begin{array}{cc}
			\rho^0\partial_t v - \nabla\cdot \left(\mathbb{S}(v)\right) = f&\quad  \text {in } (0,T) \times \Omega, \vspace{0.15cm} \\
				v_n = 0,\quad \quad \left[2\mu\mathbb{D}(v)n +\beta_0 v\right] \cdot \tau^j=0, \ \ j=1,2, & \text{ on } (0,T)\times\Gamma_{0},  \vspace{0.15cm}\\
			\left(v-\partial_t \eta e_3 \right)_{n}=0,\quad \left[2\mu\mathbb{D}(v)n +\beta_1 (v-\partial_t\eta e_3)\right] \cdot \tau^j = k_j, \ \ j=1,2, &\text{ on } (0,T)\times\Gamma (\eta^0),  \vspace{0.15cm}\\
			 v(0)=v^0 & \text{ in } \Omega.
		\end{array}
		\right.
	\end{equation}
First, we construct the lifting $w$ that is defined by $$w(t,\cdot)=\theta(\cdot) \partial_t\eta(t,\cdot)e_3,
$$
where $\theta$ is given in \eqref{theta}. Considering $u=v-w$ leads to prove the existence of strong solution to  the following system
\begin{equation}
\label{LameH}
\left\{
\begin{array}{cc}
\rho^0\partial_t u - \nabla\cdot \left(\mathbb{S}(u)\right) = \mathfrak{F} &\text {in } (0,T) \times \Omega,  \vspace{0.15cm} \\
u_n = 0 \quad  & \text{ on } (0,T)\times\partial\Omega,  \vspace{0.15cm}\\
\left[2\mu\mathbb{D}(u)n +\beta u \right] \cdot \tau^j = \mathfrak{K}_j \ \ j=1,2, & \text{ on } (0,T)\times\partial\Omega,  \vspace{0.15cm}\\
u(0)=u^0 & \text{ in } \Omega,
\end{array}
\right.
\end{equation}
where $\beta$ is defined as
$$
\beta=
\left\{
\begin{array}{cc}
	\beta_0 & \text{ on } \Gamma_{0},\\
	\beta_1 & \text{ on } \Gamma(\eta^0).\\
\end{array}\right.
$$  
The unique existence of solution to \eqref{LameH} is proved in \cite[Theorem 7.2]{zbMATH06298655} for 
$$
\mathfrak{F}\in L^p(0,T;[L^q(\Omega)]^3),\quad \mathfrak{K}=(\mathfrak{K}_1,\mathfrak{K}_2)\in L^p(0,T;[W^{1,q}(\Omega)]^2)\cap W^{1,p}(0,T;[\mathscr{W}^{-1,q}(\Omega)]^2),
$$
where $\mathfrak{K}$ is considered as a function in $\Omega$ with time derivative in $\mathscr{W}^{-1,q}(\Omega)$. This avoids using a nonlocal operator in time (see \cite{zbMATH06436098}).
Thus, we get the existence of strong solution to the system \eqref{heat}.
\begin{Theorem}
Suppose that the condition  \eqref{1635} holds and let 
$$f\in L^p(0,T; [L^q(\Omega)]^3),\quad k:=(k_1,k_2)\in L^p(0,T;[W^{1,q}(\Omega)]^2)\cap W^{1,p}(0,T;[\mathscr{W}^{-1,q}(\Omega)]^2), 
$$
$$ v^0\in [B_{pq}^{2(1-1/p)}(\Omega)]^3 \, \textnormal{ satisfying the boundary conditions} \, \eqref{initialD}_4,$$ where we have extended $k_j$, $j=1,2$, in $\Omega$. Moreover,  assume that the function $\eta$ is given as in Theorem \ref{Theorem-eta-sol}. Then, the system \eqref{heat} admits a unique solution $v\in \mathcal{X}_v$ such that
\begin{multline}
\label{1010}
\|v\|_{\mathcal{X}_v}\leq C\bigg( \|v^0\|_{[B_{pq}^{2(1-1/p)}(\Omega)]^3}+ \|f\|_{L^p(0,T;[L^q(\Omega)]^3)}+ \|\eta\|_{\mathcal{X}_\eta}\\+\|k\|_{L^p(0,T;[W^{1,q}(\Omega)]^2)\cap W^{1,p}(0,T;[\mathscr{W}^{-1,q}(\Omega)]^2)}\bigg). 
\end{multline}
\end{Theorem}


	\section{Existence and uniqueness of solution} 
	\label{fp} 
    This section is devoted to prove the existence and uniqueness of strong solution to the system \eqref{fs1}--\eqref{ic1}, equivalently to the  fluid-structure interaction system \eqref{fs}--\eqref{ic}.
    
	\subsection{Local-in-time existence}\label{Sec-Local-exis}

    The proof is based on the  Second Schauder fixed-point theorem (see, for instance, \cite[Corollary 9.7, p. 452]{Zeidler-Vol-1}.
    \begin{Theorem}[Second Schauder Fixed-Point Theorem (1927)] \label{Thm-Fixed-Point}
       Let $E$ be a separable reflexive Banach space and  $K$ is a nonempty, closed, bounded and  convex subset of $E$. Let $\Lambda$ be a self-map on $K$. If in addition  $\Lambda$ satisfies the weak continuity property:
       \begin{align*}
        y_n \to y \ \text{ weakly in } E \ \  \Longrightarrow \ \ \Lambda(y_n) \to \Lambda(y) \ \text{ weakly in } E ,
       \end{align*}
       then, $\Lambda$ admits at least one fixed point.  
    \end{Theorem}
  
Let  $(\widehat{v},\widehat{\eta})$ be fixed and let $\widehat{\rho} = \widehat\rho(\widehat{v},\widehat{\eta})$ to the solution of the linear transport equation 
    \begin{align}\label{Eq-transport-inside-FixedPt}
    \left\{
    \begin{array}{cc}
\partial_t\widehat{\rho}+\frac{1}{\delta_{\widehat{X}}}(\widehat{v}-\partial_t\widehat{X})\nabla \widehat{\rho} = -\frac{\widehat{\rho}}{\delta_{\widehat{X}}}\nabla\cdot\widehat{v} &\quad  
      \text{in } (0,T) \times \Omega, \\
      \widehat{\rho}(0)=\rho^0 &\quad  \text{in } \Omega.
    \end{array}
    \right.	
\end{align}              
Let us consider the map 
	$$
	 \Lambda:(\widehat{v},\widehat{\eta}) \longrightarrow (v,\eta),
	$$
 that associates  the unique solution $(v,\eta)$ to the linear problem \eqref{fs1*}--\eqref{bc1*} 
 for given  $(\widehat{v},\widehat{\eta})$ 
 with $\widehat{\rho} = \widehat{\rho}(\widehat{v},\widehat{\eta})$ and initial data \eqref{initialD}.
	Let us consider the norms
	 $$
	\|(v,\eta)\|_{\mathcal{X}}=\|v\|_{\mathcal{X}_v}+\|\eta\|_{\mathcal{X}_\eta},  
	$$
	$$
	\|(v^0,\eta^0,\eta^1)\|_{\mathcal{X}^0}=\|v^0\|_{[B_{qp}^{2(1-1/p)}(\Omega)]^3}+ \|\eta^0\|_{B_{qp}^{2(2-1/p)}(\omega)}+ \|\eta^1\|_{B_{qp}^{2(1-1/p)}(\omega)} 
	$$
	and,  define the set
\begin{align*}
	\mathcal{B}_{R,T}=\Big\{(\widehat{v},\widehat{\eta}) \, : \, \, \widehat{v}_n=0\ \text{ on }\Gamma_{0},\ (\widehat{v}-\partial_t\widehat{\eta}e_3)_n=0\ \text{ on } \Gamma(\eta^0), \  \|(\widehat{v},\widehat{\eta})\|_{\mathcal{X}}\leq R\Big\},
\end{align*}

To apply Theorem \ref{Thm-Fixed-Point} for the local-in-time existence of strong solution, we need to check that
\begin{itemize}
\item $\Lambda$ maps $\mathcal{B}_{R,T}$ to itself for $T$ small enough,
and 
\item $\Lambda$ is weakly continuous. 
\end{itemize}

Let $(\rho^0,v^0,\eta^0,\eta^1)$ satisfying \eqref{initialD} and $(\widehat{v},\widehat{\eta})$ be any function in $\mathcal{B}_{R,T}$. The corresponding solution $(v,\eta)$ of the linear problem \eqref{fs1*}--\eqref{bc1*} verifies the estimate
 
   \begin{multline*}
		\|(v,\eta)\|_{\mathcal X}\leq C\bigg(	\|(v^0,\eta^0,\eta^1)\|_{\mathcal{X}^0}+\|\mathbb{F}_{\widehat{X}}(\widehat{\rho},\widehat{v},\widehat{\eta})\|_{L^p(0,T;[L^q(\Omega)]^3)} +
        \|\mathbb{G}_{\widehat{X}}(\widehat{\rho},\widehat{v},\widehat{\eta})\|_{L^p(0,T;[L^q(\Omega)]^3)}
        \\+ \|\mathbb{K}_{\widehat{X}}(\widehat{\rho},\widehat{v},\widehat{\eta})\|_{L^p(0,T;[W^{1,q}(\Omega)]^2)\cap W^{1,p}(0,T;[\mathscr{W}^{-1,q}(\Omega)]^2)} \bigg),  
\end{multline*}
where  $\mathbb F_X$ and $\mathbb K_X$ are defined in  \eqref{F} and \eqref{K} respectively. 
We further   assume that $T<1$ and $R$ is sufficiently large to have
$$
C\|(v^0,\eta^0,\eta^1)\|_{\mathcal{X}^0}\leq R/2.
$$

\vspace{.5cm}

$\bullet$ 
We now  show that for $T$ small enough, $\Lambda$ maps $\mathcal{B}_{R,T}$ into itself. Then, it demands to obtain that
\begin{align}
	\label{1731}
	C\bigg( &\|\mathbb{F}_{\widehat{X}}(\widehat{\rho},\widehat{v},\widehat{\eta})\|_{L^p(0,T;[L^q(\Omega)]^3)} +\|\mathbb{G}_{\widehat{X}}(\widehat{\rho},\widehat{v},\widehat{\eta})\|_{L^p(0,T;[L^q(\Omega)]^3)}  
   \notag  \\
   &+ \|\mathbb{K}_{\widehat{X}}(\widehat{\rho},\widehat{v},\widehat{\eta})\|_{L^p(0,T;[W^{1,q}(\Omega)]^2)\cap W^{1,p}(0,T;[\mathscr{W}^{-1,q}(\Omega)]^2)}\bigg) \leq R/2,
\end{align}
 for $T$ small enough. 
 
 
We now  prove the inequality  \eqref{1731}. In the sequel, the constants $C$ will certainly  depend on $R$. First, we note that the following estimates hold
	$$
\|\delta_{\widehat{X}}-1\|_{L^\infty(0,T;W^{1,q}(\Omega))}\leq C\|\widehat{\eta}-\eta^0\|_{L^\infty(0,T;W^{1,q}(\omega))}\leq CT \|\partial_t\widehat{\eta}\|_{L^\infty(0,T;W^{1,q}(\omega))}\leq C T . 
	$$
	From the regularity of $\widehat{\eta}$, we have also
	$$
	\|\widehat{\eta}-\eta^0\|_{L^\infty(0,T;W^{2,q}(\omega))}\leq C T^{1/p'}.
	$$
	Hence,
	\begin{multline}
		\label{1327}
		\|\widehat{X}-\mathbb{I}_3\|_{L^\infty(0,T;[W^{2,q}(\Omega)]^3)}\leq C T^{1/p'},\quad \|b_{\widehat{X}}-\mathbb{I}_3\|_{L^\infty(0,T;[W^{1,q}(\Omega)]^9)}\leq CT^{1/p'},\\
		 \left\|\frac{\partial (b_{\widehat{X}})_{ij}}{\partial x_k}\right\|_{L^\infty(0,T;[L^q(\Omega)]^9)}\leq CT^{1/p'}.
	\end{multline}
 
\begin{Lemma}\label{Lemma-transport-inside-Fixedpt} Let $\widehat \rho$ be the unique solution to the linearized transport equation \eqref{Eq-transport-inside-FixedPt}. Then, we have  
    
    \begin{equation}
    \label{1125}
\|\widehat{\rho}\|_{L^\infty(0,T;W^{1,q}(\Omega))} \leq C(1+T^{1/p'}), 
    \end{equation}
and
\begin{equation}
\label{1126}
  \|\widehat{\rho}-\rho^0\|_{L^\infty(0,T;L^\infty(\Omega))} \leq C\left( T^\alpha + T^{1/p'}\right),\quad \alpha = 1-3/q.   
\end{equation}
\end{Lemma}
\begin{proof}
Using the classical estimate \eqref{1604}, we get
\begin{equation*}
\|\widehat{\rho}(t)\|_{W^{1,q}(\Omega)} \leq C\|\rho^0\|_{W^{1,q}(\Omega)}  + \int_0^t\left\|\frac{\widehat{\rho}(s)}{\delta_{\widehat{X}}}\nabla\cdot \widehat{v} (s)\right\|_{W^{1,q}(\Omega)} \ ds. 
\end{equation*}
Using Gr\"onwall inequality, we get \eqref{1125}. On the other hand,  rewriting the right-hand side of the transport equations as 
$$-\frac{\widehat \rho}{\delta_{\widehat X}}\nabla\cdot \widehat v=-\frac{1}{\delta_{\widehat{X}}}(\widehat{\rho}-\rho^0)\nabla\cdot\widehat{v}-\frac{\rho^0}{\delta_{\widehat{X}}}\nabla\cdot\widehat{v}.$$
Since
$$
\left\|\frac{\rho^0}{\delta_{\widehat{X}}}\nabla\cdot\widehat{v}\right\|_{L^1(0,T;L^\infty(\Omega))}\leq CT^{1/p'},
$$
using \eqref{1511} we get
\begin{equation}
	\|\widehat{\rho}(t)-\rho^0\|_{L^\infty(\Omega)} \leq C (T^\alpha+T^{1/p'}) + \int_0^t\left\|\frac{1}{\delta_{\widehat{X}}}(\widehat{\rho}(s)-\rho^0)\nabla\cdot \widehat{v} (s)\right\|_{L^\infty(\Omega)} \ ds.
\end{equation}
Using Gr\"onwall inequality again, we obtain \eqref{1126}.
\end{proof}   
	Let us now estimate $\mathbb{F}_{\widehat{X}}$ defined in \eqref{F}. First, we deal with the term $ \mathbb{S}_{\widehat{X}}(\widehat{v})-\mathbb{S}(\widehat{v})$ that contains terms such as
	$$
	\left( (b_{\widehat{X}})_{ik}^\top-\delta_{ik}\right) \frac{(b_{\widehat{X}})_{lm}^\top}{\delta_{\widehat{X}}}(\nabla\widehat{v})_{kl}(b_{\widehat{X}})_{mj},\quad \frac{\partial (b_{\widehat{X}}^{-\top})_{ik}}{\partial x_m}(\widehat{X})\widehat{v}_k(b_{\widehat{X}})_{mj}.
	$$
	Since $\nabla\widehat{v}\in L^p(0,T;[W^{1,q}(\Omega)]^9)$,  using \eqref{1327}, we obtain 
	\begin{equation*}
		\left\|\left( (b_{\widehat{X}})_{ik}^\top-\delta_{ik}\right) \frac{(b_{\widehat{X}})_{lm}^\top}{\delta_{\widehat{X}}}(\nabla\widehat{v})_{kl}(b_{\widehat{X}})_{mj}\right\|_{L^p(0,T;W^{1,q}(\Omega))}\leq CT^{1/p'}. 
	\end{equation*}
Now, since  $\widehat{\eta}\in C([0,T]; W^{3,q}(\omega))$, $\frac{\partial (b_{\widehat{X}}^{-\top})_{ik}}{\partial x_m}(\widehat{X})\in L^\infty(0,T;W^{1,q}(\Omega))$, and    $\widehat{v}\in C([0,T];[W^{1,q}(\Omega)]^3)$, we obtain
	\begin{equation}
		\left\|\frac{\partial (b_{\widehat{X}}^{-\top})_{ik}}{\partial x_m}(\widehat{X})\widehat{v}_k(b_{\widehat{X}})_{mj}\right\|_{L^p(0,T;W^{1,q}(\Omega))}\leq CT^{1/p}.
	\end{equation}
	This yields that 
	\begin{equation}
		\label{1704}
		\|\nabla \cdot (\mathbb{S}_{\widehat{X}}(\widehat{v})-\mathbb{S}(\widehat{v}))\|_{L^p(0,T;[L^q(\Omega)]^3)}\leq C(T^{1/p}+T^{1/p'}).
	\end{equation}
	Let us  estimate the time derivative terms:
	\begin{equation}
		\label{1422}
		\widehat{\rho}(\nabla \widehat{X}) \partial_t \widehat{v}-\rho^0\partial_t \widehat{v},\quad (\nabla \widehat{X})(\nabla \widehat{v})\partial_t \widehat{Y}(\widehat{X}),\quad \delta_{\widehat{X}}\partial_t (b_{\widehat{X}}^{-\top})(\widehat{X}) \widehat{v}.
	\end{equation}
For  the first term, using \eqref{1126} we have 
	\begin{align}
		\label{1702}
		&\|\widehat{\rho}(\nabla \widehat{X}) \partial_t \widehat{v}-\rho^0\partial_t \widehat{v}\|_{L^p(0,T;[L^q(\Omega)]^3)} \notag \\ 
        & \leq  \|\widehat{\rho}-\rho^0\|_{L^\infty(0,T;L^\infty(\Omega))}\|\nabla \widehat{X} \|_{L^\infty(0,T;[W^{1,q}(\Omega)]^9)}\|\partial_t \widehat{v}\|_{L^p(0,T;[L^q(\Omega)]^3)} \\
        & \quad +
		\|\rho^0\|_{W^{1,q}(\Omega)}\|\nabla \widehat{X}-\mathbb{I}_3 \|_{L^\infty(0,T;[W^{1,q}(\Omega)]^9)}\|\partial_t \widehat{v}\|_{L^p(0,T;[L^q(\Omega)]^3)}\notag \\
        & \leq  C(T^\alpha+T^{1/p'}).
	\end{align}  
	The other two terms in  \eqref{1422} are in $L^\infty(0,T;[L^q(\Omega)]^3)$, and thus,
	\begin{equation}
		\label{1707}
		\| (\nabla \widehat{X})(\nabla \widehat{v})\partial_t \widehat{Y}(\widehat{X})\|_{L^p(0,T;[L^q(\Omega)]^3)}+\| \delta_{\widehat{X}}\partial_t (b_{\widehat{X}}^{-\top})(\widehat{X}) \widehat{v}\|_{L^p(0,T;[L^q(\Omega)]^3)}\leq C T^{1/p}.
	\end{equation}
	Next,  we need to handle the convective term $\mathbb{N}_{\widehat{X}}$ (appearing in the definition \eqref{F}) that contains especially the terms like 
	$$
	\widehat{\rho}\frac{\partial (b_{\widehat{X}}^{-\top})_{ik}}{\partial x_j}(\widehat{X}) \frac{\partial \widehat{X}_j}{\partial y_m} \widehat{v}_k \widehat{v}_m
	,\quad \widehat{\rho}(b_{\widehat{X}}^{-\top})_{i,k}(\widehat{X})\frac{\partial \widehat{v}_k}{\partial y_l} \widehat{v}_l.
	$$
	We see that the terms above are in $L^\infty(0,T;[L^q(\Omega)]^3)$, hence
	\begin{equation}
		\label{1655}
		\|\mathbb{N}_{\widehat{X}}(\widehat{\rho},\widehat{v},\widehat{\eta})\|_{L^p(0,T;[L^q(\Omega)]^3)}\leq CT^{1/p}. 
	\end{equation}
	Finally, we estimate the pressure term
	$\gamma \widehat{\rho}^{\gamma-1}b_{\widehat{X}} \nabla \widehat{\rho}$,
	which  is in $L^\infty(0,T;[L^q(\Omega)]^3)$, and  thus,
	\begin{equation}
		\label{1700}
		\|\gamma \widehat{\rho}^{\gamma-1}b_{\widehat{X}} \nabla \widehat{\rho}\|_{L^p(0,T;[L^q(\Omega)]^3)}\leq C T^{1/p}.
	\end{equation}
	Gathering \eqref{1704}, \eqref{1702}, \eqref{1707}, \eqref{1655} and \eqref{1700}, we obtain that
	\begin{equation}
		\label{1710}
		\|\mathbb{F}_{\widehat{X}}(\widehat{\rho},\widehat{v},\widehat{\eta})\|_{L^p(0,T;[L^q(\Omega)]^3)}\leq C (T^{1/p}+T^{\alpha}+ T^{1/p'}).
	\end{equation}

	Hereafter, by using the trace theorem (see Remark \ref{1051}) with the embedding $W^{1-1/q,q}(\partial\Omega)\hookrightarrow L^\infty(\partial\Omega)$, we obtain 
	$$
	\|\mathbb{G}_{\widehat{X}}(\widehat{\rho},\widehat{v},\widehat{\eta})\|_{L^p(0,T;L^q(\omega))}\leq CT^{1/p}.
	$$
    
	Finally, we  deal with the boundary nonlinear term $\mathbb K$ defined in \eqref{K} that will be seen as a function in $\Omega$. Note that we already have 
	$$
	\|[\mathbb{K}_{\widehat{X}}(\widehat \rho, \widehat{v},\widehat{\eta})]_j\|_{L^p(0,T;W^{1,q}(\Omega))}\leq C(T^{1/p}+T^{1/p'}).
	$$
	Now we shall estimate $\partial_t[\mathbb{K}_{\widehat{X}}]_j$ in $L^p(0,T;\mathscr{W}^{-1,q}(\Omega))$ where $\mathscr{W}^{-1,q}(\Omega)$ is defined in \eqref{1213}.  We rewrite $[\mathbb{K}_{\widehat{X}}]_j$ explicitly
    \begin{align} \label{1119}
    & [\mathbb{K}_{\widehat{X}}(\widehat{\rho},\widehat{v},\widehat{\eta})]_j \notag \\
    & 
=\mu\left[\left(\nabla \widehat{v}-\frac{|N|}{|N_{\widehat{\eta}}|}b_{\widehat{X}}^{-\top}\nabla \widehat {v}\nabla \widehat {Y}(\widehat {X})b_{\widehat{X}}\right)+ \left((\nabla \widehat{v})^\top - \frac{|N|}{|N_{\widehat{\eta}}|}b_{\widehat{X}}^{-\top}(\nabla \widehat{Y}(\widehat{X}))^\top (\nabla \widehat{v})^\top b_{\widehat{X}}^{-1}b_{\widehat{X}}\right)\right]n\cdot \tau^j \notag \\
        &-\mu 
		\sum_{m,k,i,l}\left(\frac{\partial (b_X^{-\top})_{ik}}{\partial x_m}(X)+\frac{\partial (b_X^{-\top})_{mk}}{\partial x_i}(X)\right) v_k (b_X)_{ml}n_l\tau_i^j \notag \\
       & +\beta_1\left[\left(\mathbb{I}_3-b_{\widehat{X}}^{-1}b_{\widehat{X}}^{-\top} \right)\widehat{v}+\left(b_{\widehat{X}}^{-1}-\mathbb{I}_3\right)\partial_t\widehat{\eta} e_3 \right] \cdot \tau^j.
    \end{align}
    For simplicity, it is enough to treat terms of the form
    \begin{equation}
        \label{1138}
    \left( \mathbb{I}_3-b_{\widehat{X}}^{-1}\right) \widehat{v}\cdot\tau^j,\quad   \frac{\partial (b_{\widehat{X}}^{-\top})_{ik}}{\partial x_m}(\widehat{X}) \widehat{v}_k(b_{\widehat{X}})_{ml},\quad (b_{\widehat{X}}^{-\top}-\mathbb{I}_3)_{ik}\frac{\partial Y_{l'}}{\partial x_m}\frac{\partial\widehat{v}_k}{\partial y_{l'}} (b_{\widehat{X}})_{ml}.  
    \end{equation}
    Let us  first consider  the term
	$$
	\left( \mathbb{I}_3-b_{\widehat{X}}^{-1}\right) \widehat{v}\cdot\tau^j,\quad j=1,2.
	$$
	From Lemma \ref{K.}, we have
	\begin{multline}
		\label{1705}
		\left\|\partial_t\left[\left( \mathbb{I}_3-b_{\widehat{X}}^{-1}\right) \widehat{v}\cdot\tau^j\right]\right\|_{L^p(0,T;\mathscr{W}^{-1,q}(\Omega))}\leq C\bigg[ \left( \int_0^T \|\partial_t(b_{\widehat{X}}^{-1})\|^p_{[L^q(\Omega)]^9} \|\widehat{v}\|^p_{[L^q(\Omega)]^3}\ dt\right)^{1/p}  \\+\left(  \int_0^T \|(b_{\widehat{X}}^{-1}-\mathbb{I}_3)\|^p_{[L^q(\Omega)]^9} \|\partial_t\widehat{v}\|^p_{[L^q(\Omega)]^3} \ dt \right)^{1/p} \bigg]. 
	\end{multline}
	But, we have $\partial_t(b_{\widehat{X}}^{-1})\in L^\infty(0,T;[L^q(\Omega)]^9)$ and $v\in L^\infty(0,T;[L^q(\Omega)]^3)$, so that 
	$$
	\left( \int_0^T \|\partial_t(b_{\widehat{X}}^{-1})\|^p_{[L^q(\Omega)]^9} \|\widehat{v}\|^p_{[L^q(\Omega)]^3}\ dt\right)^{1/p} \leq CT^{1/p}.
	$$
	Using \eqref{1327} the second term in the right hand side of \eqref{1705} gives
	$$
	\left(  \int_0^T \|(b_{\widehat{X}}^{-1}-\mathbb{I}_3)\|^p_{[L^q(\Omega)]^9} \|\partial_t\widehat{v}\|^p_{[L^q(\Omega)]^3} \ dt \right)^{1/p} \leq CT^{1/p'}.
	$$
	Therefore, 
	\begin{equation}
		\label{1718}
		\left\|\partial_t\left[\left( \mathbb{I}_3-b_{\widehat{X}}^{-1}\right) \widehat{v}\cdot\tau^j\right]\right\|_{L^p(0,T;\mathscr{W}^{-1,q}(\Omega))}\leq C(T^{1/p}+T^{1/p'}).
	\end{equation}
    
Now, we look at the term related to $\mathbb S_{\widehat X}(\widehat v)$   in \eqref{K}. Mainly,  we need to work with the terms of the form
	$$
\frac{\partial (b_{\widehat{X}}^{-\top})_{ik}}{\partial x_m}(\widehat{X}) \widehat{v}_k(b_{\widehat{X}})_{ml},\quad (b_{\widehat{X}}^{-\top}-\mathbb{I}_3)_{ik}\frac{\partial Y_{l'}}{\partial x_m}\frac{\partial\widehat{v}_k}{\partial y_{l'}} (b_{\widehat{X}})_{ml}.
	$$
Using \eqref{1550}, we have 
\begin{align*}
& \left\|\partial_t\left[ \frac{\partial (b_{\widehat{X}}^{-\top})_{ik}}{\partial x_m}(\widehat{X}) \widehat{v}_k(b_{\widehat{X}})_{ml}\right] \right\|_{L^p(0,T;\mathscr{W}^{-1,q}(\Omega))}  \\ 
&\leq C\bigg[\left(  \int_0^T\|\partial_tb_{\widehat{X}}^{-\top}\|_{[L^q(\Omega)]^9}^p\|\widehat{v}\|_{[W^{1,q}(\Omega)]^3}^p \|b_{\widehat{X}}\|_{[W^{1,q}(\Omega)]^9}^p \ dt\right)^{1/p} 
\\ 
& \ \ + \left(  \int_0^T\left\| \frac{\partial (b_{\widehat{X}}^{-\top})_{ik}}{\partial x_m}\right\|_{L^q(\Omega)}^p  \|\partial_t\left( \widehat{v}_k(b_{\widehat{X}})_{ml}\right) \|_{L^q(\Omega)}^p\ dt\right)^{1/p}\bigg] \\
& \leq C(T^{1/p}+T^{1/p'}) 
\end{align*}
and 
\begin{align*}
& \left\|\partial_t\left[ (b_{\widehat{X}}^{-\top}-\mathbb{I}_3)_{ik}\frac{\partial Y_{l'}}{\partial x_m}\frac{\partial\widehat{v}_k}{\partial y_{l'}} (b_{\widehat{X}})_{ml}\right] \right\|_{L^p(0,T;\mathscr{W}^{-1,q}(\Omega))} \\ 
& \leq C\bigg[\left(  \int_0^T\|\partial_t\widehat{v}\|_{[L^q(\Omega)]^3}^p\|(b_{\widehat{X}}^{-\top}-\mathbb{I}_3)_{ik}\frac{\partial Y_{l'}}{\partial x_m}(b_{\widehat{X}})_{ml}\|_{W^{1,q}(\Omega)}^p \ dt\right)^{1/p}
\\
& \ \ + \left(  \int_0^T\|\nabla v\|_{[L^q(\Omega)]^9}^p \|\partial_t\left( (b_{\widehat{X}}^{-\top}-\mathbb{I}_3)_{ik}\frac{\partial Y_{l'}}{\partial x_m}(b_{\widehat{X}})_{ml}\right)\|_{L^q(\Omega)}^p\ dt\right)^{1/p}\bigg] \\
& \leq C(T^{1/p}+T^{1/p'}).
\end{align*} 
Thus, we can say  that
\begin{align}\label{K-time-derivative}
\|[\mathbb{K}_{\widehat{X}}(\widehat \rho, \widehat{v},\widehat{\eta})]_j\|_{W^{1,p}(0,T;\mathscr{W}^{-1,q}(\Omega))}\leq C(T^{1/p}+T^{1/p'}).
\end{align}

This yields that the mapping $\Lambda$ is 
well-defined and it  maps $\mathcal{B}_{R,T}$ into itself provided  $T$ is small enough.

\vspace{.5cm}
$\bullet$ Let us show the weak continuity of $\Lambda$ with respect to  the norm

\begin{equation*}
\|(v,\eta)\|_{\mathcal{X}}=\|v\|_{ L^p(0,T;[W^{2,q}(\Omega)]^3)\cap W^{1,p}(0,T; [L^q(\Omega)]^3)}+\|\eta\|_{ L^p(0,T;W^{4,q}(\Omega))\cap W^{2,p}(0,T;L^q(\omega))}. 
\end{equation*} 
Suppose that $(\widehat{v}_m,\widehat{\eta}_m)_m \subset \mathcal{B}_{R,T}$ converges weakly to $(\widehat{v},\widehat{\eta})$ in norm $\|\cdot\|_{\mathcal{X}}$. We show that $(v_m,\eta_m)= \Lambda(\widehat{v}_m,\widehat{\eta}_m)_m$ converges weakly to $\Lambda(\widehat{v},\widehat{\eta})=(v,\eta)$ in norm $\|\cdot\|_{\mathcal{X}}$.
We set $\widehat{\rho}_m = \widehat \rho_m (\widehat{v}_m,\widehat{\eta}_m)$ that converges weakly to some $\widehat{\rho}$ in norm $\|\cdot\|_{\mathcal{X}_\rho}$.

In particular,  we have 
$$
\widehat{\rho}_m\longrightarrow \widehat{\rho} \, \text{       strongly in } L^\infty(0,T;L^q(\Omega)),\quad \widehat{v}_m\longrightarrow \widehat{v}\,  \text{ strongly in } L^\infty(0,T;[L^q(\Omega)]^3).
	$$
Since $(v_m,\eta_m)_m \subset \mathcal{B}_{R,T}$, there exists some $(v^\star, \eta^\star)$ such that,  we have 
	$$
	v_m\longrightarrow v^\star \text{ weakly in } W^{1,p}(0,T;[L^q(\Omega)]^3)\cap L^p(0,T;[W^{2,q}(\Omega)]^3)\text{ and weakly-* in } L^\infty(0,T;[W^{1,q}(\Omega)]^3),
	$$
	$$
	\eta_m\longrightarrow \eta^\star \text{ weakly in } W^{2,p}(0,T;L^q(\omega))\cap L^p(0,T;W^{4,q}(\omega))\text{ and weakly-* in } L^\infty(0,T;W^{3,q}(\omega)),
	$$
	$$
	\partial_t\eta_m\longrightarrow \partial_t\eta^\star \text{ weakly in } W^{1,p}(0,T;L^q(\omega))\cap L^p(0,T;W^{2,q}(\omega))\text{ and weakly-* in } L^\infty(0,T;W^{1,q}(\omega)) . 
	$$
	Let us  pass to the limit in the continuity equation. We have 
	$$
	\partial_t\widehat{\rho}_m\longrightarrow \partial_t\widehat{\rho} \text{ weakly in } L^p(0,T;L^q(\Omega)).
	$$
	Now consider the term $\frac{1}{\delta_{\widehat{X}_m}}(\partial_t\widehat{X}_m -\widehat{v}_m)\cdot \nabla\widehat{\rho}_m$. 
	We have that $\nabla \widehat{\rho}_m \longrightarrow \nabla\widehat{\rho}$ weakly-* in $L^\infty(0,T;L^q(\Omega))$ and since $W^{2,q}(\Omega)$ is compactly embedded in $W^{1,q}(\Omega)$,  using Aubin-Lions lemma we get $\widehat{v}_m \longrightarrow \widehat{v}$ strongly in $L^p(0,T;[W^{1,q}(\Omega)]^3)$. Thus, we can write for any $w\in L^{p'}(0,T;L^{q'}(\Omega))$,
	\begin{align*}
&\int_0^T\int_\Omega\left(\frac{1}{\delta_{\widehat{X}_m}}(\partial_t\widehat{X}_m -\widehat{v}_m )\cdot\nabla\widehat{\rho}_m -\frac{1}{\delta_{\widehat{X}}}(\partial_t\widehat{X} - \widehat{v})\cdot\nabla\widehat{\rho}\right) w\, dydt \\
& = \int_0^T\int_\Omega \left(\frac{1}{\delta_{\widehat{X}_m }}-\frac{1}{\delta_{\widehat{X}}}\right)(\partial_t\widehat{X}_m -\widehat{v}_m )\cdot\nabla\widehat{\rho}_m w \, dydt \\
   & +\int_0^T\int_\Omega\frac{1}{\delta_{\widehat{X}}}(\partial_t\widehat{X}_m -\widehat{v}_m -(\partial_t\widehat{X}-\widehat{v}))\cdot\nabla\widehat{\rho}_m w \, dydt+\int_0^T\int_\Omega\frac{1}{\delta_{\widehat{X}}}(\partial_t\widehat{X}-\widehat{v})\cdot(\nabla\widehat{\rho}_m -\nabla\widehat{\rho})w \, dydt.
	\end{align*}
Thanks to Aubin-Lions compactness argument, $\widehat{\eta}_m\longrightarrow \widehat{\eta}$ strongly in $L^p(L^\infty)$. To deal with the first term, we recall that
$$
\frac{1}{\delta_{\widehat{X}_m}}-\frac{1}{\delta_{\widehat{X}}}= \frac{\delta_{\widehat{X}}-\delta_{\widehat{X}_m}}{\delta_{\widehat{X}_m}\delta_{\widehat{X}}} = \frac{\partial_{y_3}\theta (\widehat{\eta}_m-\widehat{\eta})}{\delta_{\widehat{X}_m}\delta_{\widehat{X}}}.
$$
We deduce that the term above tends to zero strongly in $L^p(L^\infty)$ and since $(\partial_t \widehat{X}_m-\widehat{v}_m)\cdot\nabla \widehat{\rho}_m$ is bounded in $L^\infty(L^q)$ the first term tends to zero. 
Next, we recall that $\partial_t \widehat{X}_m=\theta\partial_t\widehat{\eta}_m e_3$. Using Aubin-Lions lemma again, we get that $\partial_t\widehat{\eta}_m \rightarrow \partial_t\widehat{\eta} \text{ strongly in } L^p(L^\infty)$. Then, since $\nabla \widehat{\rho}_m$ is bounded in $L^\infty(L^q)$, we get that the second term tends to zero. Finally, since $\nabla\widehat{\rho}_m$ converges weakly-* to $\nabla\widehat{\rho}$ in $L^\infty(L^q)$, the third term goes to zero. Then, by uniqueness of the solution of the transport equation, we get $\widehat{\rho} = \widehat{\rho}(\widehat{v},\widehat{\eta})$.

Let us now pass to the limit in the parabolic part of the system. We have by Aubin-Lions lemma that
	$$
	\widehat{\eta}_m\longrightarrow \widehat{\eta}\text{ strongly in } L^p(0,T;W^{3,q}(\omega)),
	$$
and 
	$$
	\widehat{\rho}_m \longrightarrow \widehat{\rho}\text{ strongly in } L^\infty(0,T;L^q(\Omega))\text{ and } \widehat{v}_m \longrightarrow \widehat{v} \text{ strongly in } L^p(0,T;[W^{1,q}(\Omega)]^3).
	$$
	Then, we can also check that 
	$$
	\mathbb{N}_{\widehat{X}_m}(\widehat{\rho}_m,\widehat{v}_m,\widehat{\eta}_m)\longrightarrow \mathbb{N}_{\widehat{X}}(\widehat{\rho},\widehat{v},\widehat{\eta}) \text{ strongly in } L^p(0,T;[L^q(\Omega)]^3),
	$$
	and 
	$$
	\mathbb{T}_{\widehat{X}_m}(\widehat{\rho}_m,\widehat{v}_m,\widehat{\eta}_m)\longrightarrow 	\mathbb{T}_{\widehat{X}}(\widehat{\rho},\widehat{v},\widehat{\eta}) \text{ weakly in } L^p(0,T;[L^q(\Omega)]^3).
	$$
	On the other hand, since 
	$$
	\widehat{\rho}_m\longrightarrow \widehat{\rho} \, \text{ strongly in } L^\infty(0,T;L^q(\Omega)),\quad \nabla\widehat{\rho}_m \longrightarrow \nabla\widehat{\rho} \, \text{ weakly in } L^p(0,T;[L^q(\Omega)]^9),
	$$
	we have
	$$
	\gamma\widehat{\rho}_m^{\gamma-1}b_{\widehat{X}_m}\nabla\widehat{\rho}_m\longrightarrow 	\gamma\widehat{\rho}^{\gamma-1}b_{\widehat{X}}\nabla\widehat{\rho} \, \text{ weakly in } L^p(0,T;[L^q(\Omega)]^3).
	$$
	Since $\nabla^2\widehat{v}_m \longrightarrow \nabla^2\widehat{v}\text{ weakly in } L^p(0,T;[L^q(\Omega)]^9)$,  then 

\begin{equation}
	\label{1250}
\nabla\cdot\mathbb{S}_{\widehat{X}_m} (\widehat{v}_m) \longrightarrow \nabla\cdot\mathbb{S}_{\widehat{X}} (\widehat{v}) \text{ weakly in } L^p(0,T;[L^q(\Omega)]^3). 
\end{equation}
Thus,
$$
\mathbb{F}_{\widehat{X}_m}(\widehat{\rho}_m,\widehat{v}_m,\widehat{\eta}_m)\longrightarrow \mathbb{F}_{\widehat{X}}(\widehat{\rho},\widehat{v},\widehat{\eta}) \text{ weakly in } L^p(0,T;[L^q(\Omega)]^3). 
$$
In particular, from \eqref{1250}, we can deduce that
$$
\mathbb{G}_{\widehat{X}_m}(\widehat{\rho}_m,\widehat{v}_m,\widehat{\eta}_m) \longrightarrow \mathbb{G}_{\widehat{X}}(\widehat{\rho},\widehat{v},\widehat{\eta}) \text{ weakly in } L^p(0,T;L^q(\omega)),
$$
and 
$$
[\mathbb{K}_{\widehat{X}_m}(\widehat{\rho}_m, \widehat{v}_m,\widehat{\eta}_m)]_j \longrightarrow [\mathbb{K}_{\widehat{X}}(\widehat \rho, \widehat{v},\widehat{\eta})]_j \text{ weakly in } L^p(0,T;W^{1,q}(\Omega)). 
$$
Finally, we show that $[\mathbb{K}_{\widehat{X}_m}(\widehat \rho_m,\widehat{v}_m,\widehat{\eta}_m)]_j$ converges weakly in $W^{1,p}(0,T;\mathscr{W}^{-1,q}(\Omega))$. Recall that the most important terms of $[\mathbb{K}_{\widehat{X}}(\widehat \rho,\widehat{v},\widehat{\eta})]_j$ are given in \eqref{1138}. In what follows, we only need to consider
\begin{multline}
        \label{1222}
    \left( \mathbb{I}_3-b_{\widehat{X}_m}^{-1}\right) \widehat{v}_m\cdot\tau^j-\left( \mathbb{I}_3-b_{\widehat{X}}^{-1}\right) \widehat{v}\cdot\tau^j,\quad   \frac{\partial (b_{\widehat{X}_m}^{-\top})_{ik}}{\partial x_r}(\widehat{v}_m)_k(b_{\widehat{X}_m})_{rj}-\frac{\partial (b_{\widehat{X}}^{-\top})_{ik}}{\partial x_r} \widehat{v}_k(b_{\widehat{X}})_{rj},\\(b_{\widehat{X}_m}^{-\top}-\mathbb{I}_3)_{ik}\frac{\partial (\widehat{Y}_m)_l}{\partial x_r}\frac{\partial(\widehat{v}_m)_k}{\partial y_l} (b_{\widehat{X}_m})_{rj}-(b_{\widehat{X}}^{-\top}-\mathbb{I}_3)_{ik}\frac{\partial \widehat{Y}_l}{\partial x_r}\frac{\partial\widehat{v}_k}{\partial y_l} (b_{\widehat{X}})_{rj}.  
    \end{multline}
Since $b_{\widehat{X}}$ and $\nabla Y$ are regular enough, we mainly  deal with the simplified term  
$$
(b_{\widehat{X}}^{-\top}-\mathbb{I})_{ik}\frac{\partial}{\partial y_l} (\widehat{v}_m - \widehat{v})_k.
$$
We focus on the most difficult term which contains the terms related to gradient of $\widehat v$. The other terms can be handled  directly. In fact, 
for all $\varphi\in L^{p'}(0,T;C_0^\infty(\mathbb{R}^3))$, we have 
\begin{align*}
&\int_{\mathbb{R}^3} \imath\partial_t\left( (b_{\widehat{X}}^{-\top}-\mathbb{I})_{ik}\frac{\partial}{\partial y_l} (\widehat{v}_m-\widehat{v})_k\right)  \cdot (\mathbb{I}-\Delta)^{-1/2}\varphi \ dy   \\
& =\int_{\mathbb{R}^3} \imath\partial_t\frac{\partial}{\partial y_l}  \left( (b_{\widehat{X}}^{-\top}-\mathbb{I})_{ik}(\widehat{v}_m-\widehat{v})_k\right)  \cdot (\mathbb{I}-\Delta)^{-1/2}\varphi \ dy \\
& \quad -\int_{\mathbb{R}^3} \imath\partial_t\left( \frac{\partial}{\partial y_l} \left(  (b_{\widehat{X}}^{-\top}-\mathbb{I})_{ik}\right)( \widehat{v}_m-\widehat{v})_k\right)   \cdot (\mathbb{I}-\Delta)^{-1/2}\varphi \ dy.
\end{align*}
 We apply then the formula \eqref{IPP} with
$$f = \imath \ \partial_t\left[(b_{\widehat{X}}-\mathbb{I}_3)(\widehat{v}_m-v)\right],
$$
we obtain
\begin{equation*}
\int_0^T\int_{\mathbb{R}^3}((I-\Delta)^{-1/2}\partial_{y_\ell} f)\,\varphi\,dy dt\\=  -
\int_0^T\int_{\mathbb{R}^3} \imath \ \partial_t\left[(b_{\widehat{X}}-\mathbb{I}_3)(\widehat{v}_m-v)\right]\cdot ((I-\Delta)^{-1/2}\partial_{y_\ell} \varphi \ dydt.
\end{equation*}
Since $\widehat{v}_m$ (resp. $\partial_t\widehat{v}_m$) converges weakly to $\widehat{v}$ (resp. $\partial_t\widehat{v}$) in $L^p(L^q)$. Then this term tends to zero. The other terms can be managed without integration by parts.  This concludes 
$$
\partial_t[\mathbb{K}_{\widehat{X}_m}(\widehat \rho_m, \widehat{v}_m,\widehat{\eta}_m)]_i \longrightarrow \partial_t[\mathbb{K}_{\widehat{X}}(\widehat \rho, \widehat{v},\widehat{\eta})]_i \text{ weakly in } L^p(0,T;\mathscr{W}^{-1,q}(\Omega)). 
$$
Then, by uniqueness  of the solution of the linear problem, we obtain that $(v^\star,\eta^\star)=(v,\eta)$, and therefore we obtain the existence result for our concerned system  for a small time interval $(0,T_0)$ with $T_0>0$.


\subsection{Uniqueness}\label{Sec-Uniquness} Let us prove the uniqueness of solution for the system \eqref{fs1}--\eqref{ic1}.
First, we define the norm 
\begin{multline*}
	\|(v,\eta)\|_{\mathcal{X}_{weak}}:=\|v\|_{L^2(0,T;[H^1(\Omega)]^3)}+ \|v\|_{L^\infty(0,T;[L^2(\Omega)]^3)}+ \|\eta\|_{L^\infty(0,T;H^2(\omega))}\\+\|\partial_t\eta\|_{L^\infty(0,T;L^2(\omega))}+\|\partial_t\eta\|_{L^2(0,T;H^1(\omega))}, 
\end{multline*}
and  consider two solutions $(\rho^1,v^1,\eta^1)$ and $(\rho^2,v^2,\eta^2)$ of the system \eqref{fs1}--\eqref{ic1}.  Consider the differences  
$$\overline{\rho}=\rho^2-\rho^1,\quad \overline{v}=v^2-v^1,\quad  \overline{\eta}=\eta^2-\eta^1 .$$
Then, our  aim is to show that
$$
\|\overline{\rho}\|_{L^\infty(0,T;L^2(\Omega))}+\|(\overline{v},\overline{\eta})\|_{\mathcal{X}_{weak}}=0 , 
$$
which ensures that  $(\overline{\rho},\overline{v},\overline{\eta})=0$ almost  everywhere, and consequently, the  uniqueness  follows.  
The density $\rho^i$ satisfies the equation
$$
\partial_t\rho^i+\frac{1}{\delta_{X^i}}(v^i-\partial_tX^i)\nabla\rho^i=-\frac{\rho^i}{\delta_{X^i}}\nabla\cdot v^i,\quad i=1,2 , 
$$
where $$X^i=X_{\eta^i}.$$
Let start with the continuity equation, 
we have 
\begin{equation*}
	\partial_t\overline{\rho}+\frac{1}{\delta_{X^2}}(v^2-\partial_tX^2)\cdot\nabla\overline{\rho}=\mathscr{H} -\frac{\overline{\rho}}{\delta_{X^2}}\nabla\cdot v^1,
\end{equation*} 
where
\begin{equation*}
	\mathscr{H}=\left( \frac{1}{\delta_{X^1}}(v^1-\partial_tX^1)-\frac{1}{\delta_{X^2}}(v^2-\partial_tX^2) \right)\cdot \nabla\rho^1 -\frac{\rho^2}{\delta_{X^2}}\nabla\cdot\overline{v}+\rho^1\left( \frac{1}{\delta_{X^1}}- \frac{1}{\delta_{X^2}}\right) \nabla\cdot v^1.
\end{equation*}
Then using \eqref{1532} from Proposition \ref{masseq}, we get
\begin{equation*}
	\|\overline{\rho}(t)\|_{L^2(\Omega)}\leq C\bigg(  \int_0^t \|\mathscr{H}\|_{L^2(\Omega)}\ dt+\int_0^t \|\delta_{X^2}^{-1}\nabla\cdot v^1(s)\|_{L^\infty(\Omega)} \|\overline{\rho}(s)\|_{L^2(\Omega)} \ ds
	\bigg). 
\end{equation*}
Note  that
\begin{multline*}
	\int_0^t \|\mathscr{H}\|_{L^2(\Omega)}\ dt\leq C\bigg(T^{1/2}\|\nabla\cdot\overline{v}\|_{L^2(0,T;L^2(\Omega))}+T \|\overline{\eta}\|_{L^\infty(0,T;L^\infty(\Omega))}\\+T^{1/2}\|\partial_t\overline{\eta}\|_{L^2(0,T;H^1(\omega))} + T^{1/2}\|\overline{v}\|_{L^2(0,T;[H^1(\Omega)]^3)} \bigg).
\end{multline*}
Then, using Gr\"{o}nwall's Lemma and H\"{o}lder inequality, we deduce that
\begin{multline*}
	\|\overline{\rho}(t)\|_{L^2(\Omega)}\leq  C\bigg(T^{1/2}\|\nabla\cdot\overline{v}\|_{L^2(0,T;L^2(\Omega))}+T \|\overline{\eta}\|_{L^\infty(0,T;L^\infty(\Omega))}\\+T^{1/2}\|\partial_t\overline{\eta}\|_{L^2(0,T;H^1(\omega))} + T^{1/2}\|\overline{v}\|_{L^2(0,T;[H^1(\Omega)]^3)} \bigg)\leq CT^{1/2}\|(\overline{v},\overline{\eta})\|_{\mathcal{X}_{weak}}.
\end{multline*}
Now, we deal with the parabolic part. We set 
$$
N^i=N_{\eta^i},\quad i=1,2.
$$
We recall that each $(v^i,\eta^i)$ satisfies the system 
\begin{equation}
	\label{1202}
	\left\{
	\begin{array}{cc}
		\frac{1}{\rho^i}\left(\mathbb{T}_{X^i}(\rho^i,v^i,\eta^i)-\nabla\cdot (\mathbb{S}_{X^i}(v^i))+\mathbb{N}_{X^i}(\rho^i,v^i)+\nabla\cdot (p(\rho^i)b^\top_{X^i})\right)=0& \text{ in } (0,T)\times \Omega,\\
		\partial_{tt}\eta^i + \Delta^2 \eta^i -\Delta \partial_t\eta^i= \mathbb{G}_{X^i}(\rho^i,v^i,\eta^i)& \text{in } (0,T)\times\omega,
	\end{array}
	\right.
\end{equation}
with the boundary conditions
\begin{equation}
	\label{1203}
	\left\{
	\begin{array}{cc}
		(v^i)_n = 0,\ \, \left[2\mu\mathbb{D}(v^i)n +\beta_0 v^i\right] \cdot \tau^j=0 & \text{ on } (0,T)\times\Gamma_{0},\\
		\left(v^i-\partial_t \eta^i e_3 \right)_{n}=0,\ \,  \left[2\mu b_{X^i}^{-1}\frac{|N|}{|N^i|}\mathbb{D}_{X^i}(v^i)n+\beta_1 b_{X^i}^{-1}(b_{X^i}^{-\top}v^i-\partial_t\eta^i e_3) \right] \cdot \tau^j=0 &\text{ on } (0,T)\times\Gamma (\eta^0) . 
	\end{array}
	\right.
\end{equation}
Then, the difference $(\overline{v},\overline{\eta})$ satisfies the following linear system
\begin{equation}
	\label{1336-1}
	\left\{
	\begin{array}{cc}
		\left( \mathbb{T}_{X^2}(\rho^2,\overline{v},\eta^2)+\rho^2\mathscr{C}_{v^2}(\overline{v})-\nabla\cdot (\mathbb{S}_{X^2}(\overline{v}))\right)  =\mathscr{F}_1+\nabla\cdot\mathscr{F}_2& \text{ in } (0,T)\times \Omega,  \vspace{0.15cm}\\
		\partial_{tt}\overline{\eta} + \Delta^2 \overline{\eta} -\Delta \partial_t\overline{\eta}= -(\mathbb{S}_{X^2}(\overline{v})) N\cdot e_3 -\mathscr{F}_2N\cdot e_3& \text{in } (0,T)\times\omega , 
	\end{array}
	\right.
\end{equation}
with  boundary conditions
\begin{equation}
	\label{1335-1}
	\left\{
	\begin{array}{cc}
		(\overline{v})_n = 0,\  \left[2\mu\mathbb{D}(\overline{v})n +\beta_0 \overline{v}\right] \cdot \tau^j=0 & \text{on } (0,T)\times\Gamma_{0},\\
		\left(\overline{v}-\partial_t \overline{\eta} e_3 \right)_{n}=0,\  \left[2\mu b_{X^2}^{-1}\frac{|N|}{|N^2|}\mathbb{D}_{X^2}(\overline{v})n+\mathscr{F}_2n+\beta_1 b_{X^2}^{-1}(b_{X^2}^{-\top}\overline{v}-\partial_t\overline{\eta} e_3) \right] \cdot \tau^j=[ \mathscr{K}]_j &\text{on } (0,T)\times\Gamma (\eta^0),
	\end{array}
	\right.
\end{equation}
and ``zero'' initial conditions. 
To make the presentation simpler, we postpone the mathematical expressions of the terms $\mathscr{F}_1$, $\mathscr{F}_2$ and $\mathscr{K}$  in Appendix \ref{calcul} by  \eqref{1443}, \eqref{1425} and \eqref{2010} respectively.  The linear term $\mathscr{C}_{v^2}(\overline{v})$ represents the convective term with respect to $v^2$ and is defined in \eqref{1020}.

We see that the system \eqref{1336-1}--\eqref{1335-1} reads  as \eqref{1028}--\eqref{1021} in Lemma \ref{1821}. 
Thus, one may apply the estimate \eqref{1022} of Lemma \ref{1821}  with zero initial conditions to  obtain
\begin{align}\label{1216}
	\|(\overline{v},\overline{\eta})\|_{\mathcal{X}_{weak}}\leq C\bigg( \|\mathscr{F}_1\|_{L^2(0,T;[H^1(\Omega)']^9)}+\|\mathscr{F}_2\|_{L^2(0,T;[L^2(\Omega)]^9)}+ \|\mathscr{K}\|_{L^2(0,T;[H^{-1/2}(\partial\Omega)]^2)}\bigg). 
\end{align} 
We need to estimate the right hand side of inequality \eqref{1216} in terms of $\|(\overline{v},\overline{\eta})\|_{\mathcal{X}_{weak}}$. From the expression of $\mathscr{F}_2$ given by \eqref{1425}, we compute by  applying H\"{o}lder inequality, that
\begin{multline*}
	\|\mathscr{F}_2\|_{L^2(0,T; [L^2(\Omega)]^9)}\leq C\bigg( T^{1/2}\|\overline{\eta}\|_{L^\infty(0,T;H^2(\omega))}\|v^1\|_{L^\infty(0,T;[L^\infty(\Omega)]^3)}\|b_{X^2}\|_{L^\infty(0,T;L^\infty(\Omega))}\\+T^{1/2}\|\partial b_{X^1}^{-\top}\|_{L^\infty(0,T;L^\infty(\omega))}\|v^1\|_{L^\infty(0,T;L^\infty(\Omega))}\|\overline{\eta}\|_{L^\infty(0,T;H^1(\omega))} \\+T^{1/2}\|\overline{\eta}\|_{L^\infty(0,T;H^2(\omega))}\|\nabla v^1\|_{L^\infty(0,T;L^q(\Omega))}+\|\overline{\rho}\|_{L^\infty(0,T;L^2(\Omega))}\|b_{X^1}^\top\|_{L^\infty(0,T;[L^\infty(\Omega)]^9)} \bigg), 
\end{multline*}
and thus,
\begin{equation}
	\label{1800}
	\|\mathscr{F}_2\|_{L^2(0,T; [L^2(\Omega)]^9)}\leq CT^{1/2}\|(\overline{v},\overline{\eta})\|_{\mathcal{X}_{weak}}.
\end{equation}
Now, from  \eqref{1443}, we have   that 
$$
\mathscr{F}_1=\mathscr{F}_{11}+\mathscr{F}_{12}, \quad \mathscr{F}_{11}=-\rho^2\nabla\cdot\mathscr{F},
$$
where $\mathscr{F}_{11}$, $\mathscr{F}$ and $\mathscr{F}_{12}$ are respectively defined  in \eqref{1531}, \eqref{1806} and \eqref{1239}.
First, we deal with $\mathscr{F}_{12}$ that can be estimated in norm $L^2(0,T;[L^{6/5}(\Omega)]^3)$ and we use the  embedding $L^{6/5}(\Omega) \hookrightarrow W^{-1,2}(\Omega)$.
Using again H\"{o}lder inequality and the fact that
\begin{equation}
	\label{1208}
	\|\nabla\overline{\eta}\|_{L^\infty(0,T;[L^2(\omega)]^2)}\leq T^{1/2}\|\nabla\partial_t\overline{\eta}\|_{L^2(0,T;[L^2(\omega)]^2)},
\end{equation}
we get
\begin{multline*}
	\|\mathscr{F}_{12}\|_{L^2(0,T;[L^{6/5}(\Omega)]^3)}\leq C\bigg( T^{1/2}\|\partial_t\overline{\eta}\|_{L^2(0,T;H^1(\omega))}\|\partial_t v^1\|_{L^2(0,T;[L^q(\Omega)]^3)} +\\ T^{1/2}\|\partial_t\overline{\eta}\|_{L^2(0,T;H^1(\omega))}\|\nabla v^1\|_{L^2(0,T;[L^\infty(\Omega)]^9)}+ T^{1/2}\|\partial_t\overline{\eta}\|_{L^\infty(0,T;L^2(\omega))}\|\nabla v^1\|_{L^\infty(0,T;[L^q(\Omega)]^9)}\\+ T^{1/2}\|\partial_t\overline{\eta}\|_{L^\infty(0,T;L^2(\omega))}\| v^1\|_{L^\infty(0,T;[L^q(\Omega)]^9)}+ T^{1/2}\|\nabla\partial_t\overline{\eta}\|_{L^2(0,T;L^2(\omega))}\| v^1\|_{L^2(0,T;[L^q(\Omega)]^9)}\\+ T^{1/2}\|\overline{v}\|_{L^\infty(0,T;[L^2(\Omega)]^3)}\| v^1\|_{L^\infty(0,T;[L^q(\Omega)]^9)}+\| \overline{\rho}\|_{L^\infty(0,T;L^2(\Omega))}\bigg).   
\end{multline*}
Thus, 
\begin{equation}
	\label{1514}
	\|\mathscr{F}_{12}\|_{L^2(0,T;[L^{6/5}(\Omega)]^3)}\leq C T^{1/2}\|(\overline{v},\overline{\eta})\|_{\mathcal{X}_{weak}}.
\end{equation}
Finally, we estimate $\mathscr{F}_{11}$ in $W^{-1,2}(\Omega)$. More precisely, we estimate $\mathscr{F}$ in $L^2(0,T;[L^2(\Omega)]^9)$ and its trace in $L^2(0,T;[L^2(\partial\Omega)]^9)$.  Recall the expression of $\mathscr{F}$ from \eqref{1806}, which is
\begin{multline*}
	\mathscr{F}=\Big[ ((b^{-\top}_{X^2}v^2-b^{-\top}_{X^1}v^1)\otimes b^{-\top}_{X^2}v^1)b_{X^2}+(b^{-\top}_{X^1}v^1\otimes (b^{-\top}_{X^2}v^1-b^{-\top}_{X^1}v^1))b_{X^2}\\+ (b^{-\top}_{X^1}v^1\otimes b^{-\top}_{X^1}v^1)(b_{X^2}-b_{X^1})+\left(\frac{1}{\rho^2}((\rho^2)^{\gamma}(b_{X^2}^\top-b_{X^1}^\top))\right)\\-\left( b^\top_{X^2}v^1 \otimes(v^2-v^1)\right)+((\partial_t X^1-\partial_t X^2)\otimes v^1)\Big].
\end{multline*}
We estimate the two main terms:
$$
A : = b_{X^1}^{-\top}v^1\otimes b_{X^1}^{-\top}v^1(b_{X^2}-b_{X^1}),\quad B := b_{X^2}^{-\top}v^1\otimes (v^2-v^1),
$$
and the other terms can be treated in a similar way. 
 Since $\|\overline{\eta}\|_{L^\infty(0,T;H^1(\omega))}\leq T^{1/2}\|\partial_t\overline{\eta}\|_{L^2(0,T;H^1(\omega))}$, then
$$
\|b_{X^2}-b_{X^1}\|_{L^2(0,T;[L^2(\Omega)]^9)}\leq T\|\partial_t\overline{\eta}\|_{L^2(0,T;[H^1(\omega)]^9)}.
$$
Moreover, $b_{X^1}^{-\top}v^1\in L^\infty(0,T;[L^\infty(\Omega)]^3)\cap L^\infty(0,T;[L^\infty(\partial\Omega)^3])$, and thus
$$
\|A\|_{L^2(0,T;[L^2(\Omega)]^9)\cap L^2(0,T;[L^2(\partial\Omega)]^9)}
\leq CT\|\partial_t\overline{\eta}\|_{L^2(0,T;H^1(\omega))}.$$ 
Now, we deal with $B$.  Since $\overline{v}\in L^\infty(0,T;[L^2(\Omega)]^3)$ so
$$
\|B\|_{L^2(0,T;[L^2(\Omega)]^9)}\leq CT^{1/2}\|\overline{v}\|_{L^\infty(0,T;[L^2(\Omega)]^3)}.
$$
We get
\begin{align}
	\label{1439}
	\|\mathscr{F}\|_{L^2(0,T; [L^2(\Omega)]^9)} & \leq CT^{1/2}\bigg(\|\overline{v}\|_{L^\infty(0,T;[L^2(\Omega)]^3)}+ \|\partial_t\overline{\eta}\|_{L^2(0,T;H^1(\omega))}+\|\partial_t\overline{\eta}\|_{L^\infty(0,T;L^2(\omega))}\bigg) \notag \\ 
	& \leq CT^{1/2}\|(\overline{v},\overline{\eta})\|_{\mathcal{X}_{weak}}.
\end{align}
Moreover, by interpolation, we have that $v\in L^3(0,T;[H^{1/6}(\partial\Omega)]^3)$. Since $H^{1/6}(\partial\Omega) \hookrightarrow L^2(\partial\Omega)$ and $ W^{1-1/q,q}(\partial\Omega)\hookrightarrow L^\infty(\partial\Omega)$ for $q>3$, we obtain
$$
\|B\|_{L^2(0,T;[L^2(\partial\Omega)]^9)}\leq CT^{1/6}\|\overline{v}\|_{L^3(0,T;[L^2(\partial\Omega)]^3)}\leq CT^{1/6}(\|\overline{v}\|_{L^2(0,T;[H^1(\Omega)]^3)}+\|\overline{v}\|_{L^\infty(0,T;[L^2(\Omega)]^3)}).
$$
Then, we deduce 
\begin{align} 
	\label{1755}
	&\|\mathscr{F}\|_{L^2(0,T; [L^2(\partial\Omega)]^9)}  \leq C\bigg(T^{1/6}\|\overline{v}\|_{L^{3}(0,T;[L^2(\partial\Omega)]^3)}+ T^{1/2}\|\partial_t\overline{\eta}\|_{L^2(0,T;H^1(\omega))} +T^{1/2}\|\partial_t\overline{\eta}\|_{L^\infty(0,T;L^2(\omega))}\bigg) \notag \\
	&\leq C\bigg(T^{1/6}\|\overline{v}\|_{L^{\infty}(0,T;[L^2(\Omega)]^3)}+T^{1/6}\|\overline{v}\|_{L^2(0,T;[H^1(\Omega)]^3)} + T^{1/2}\|\partial_t\overline{\eta}\|_{L^2(0,T;H^1(\omega))}+T^{1/2}\|\partial_t\overline{\eta}\|_{L^\infty(0,T;L^2(\omega))}\bigg) \notag 
	\\
	&\leq CT^{1/6}\|(\overline{v},\overline{\eta})\|_{\mathcal{X}_{weak}}.
\end{align}

Hence, combining \eqref{1755} and \eqref{1439}, we get 
\begin{equation}
	\label{1648}
	\|\mathscr{F}_{11}\|_{L^2(0,T;[H^1(\Omega)']^3)} \leq CT^{1/6}\|(\overline{v},\overline{\eta})\|_{\mathcal{X}_{weak}}.
\end{equation}

Finally, recall $\mathscr{K}$ from \eqref{2010}.  We notice that $[\mathscr{K}]_j$ is in $L^2(0,T;L^2(\partial\Omega))$. Indeed, $\mathbb{D}_{X^1}(v^1)\in L^p(0,T;[W^{1-1/q,q}(\partial\Omega)]^9)$, since $W^{1-1/q,q}(\partial\Omega) \hookrightarrow L^\infty(\partial\Omega)$, we have $\mathbb{D}_{X^1}(v^1)\in L^p(0,T;[L^\infty(\partial\Omega)]^9)$.  Then using \eqref{1208}, we get
\begin{equation}
	\label{2037}
	\|\mathscr{K}_j\|_{L^2(0,T;L^2(\partial\Omega))}\leq CT^{1/2}\|\partial_t\overline{\eta}\|_{L^2(0,T; H^1(\omega))},\quad j=1,2. 
\end{equation}
Using \eqref{1800},  \eqref{1514}, \eqref{1648} and \eqref{2037}, we get that for some $\gamma>0$
$$
\|(\overline{v}, \overline{\eta})\|_{\mathcal{X}_{weak}}\leq CT^\gamma \|(\overline{v}, \overline{\eta})\|_{\mathcal{X}_{weak}}.
$$
Taking $T=T_1$ small enough, we obtain that $\|(\overline{v}, \overline{\eta})\|_{\mathcal{X}_{weak}}=0$. Repeating the procedure over $(mT_1;(m+1)T_1)$ for $m\in \mathbb{N}$, we get then the uniqueness of the solution over $(0,T_0)$. 

\vspace{.2cm}

  This concludes the proof of the main result of our paper. 

\vspace{.5cm}
\noindent
{\bf Acknowledgments.} The works of K.B. and \v{S}.N. have been supported by the  Praemium Academiae of \v{S}.N, and the Institute
of Mathematics, CAS is supported by RVO:67985840.

\appendix

\section{Some useful results}\label{Apndx-A} 

\crefname{section}{appendice}{appendices}
\Crefname{section}{Appendice}{Appendices}

We have the following classical result for the mass equation
\begin{equation}
	\label{mass}
	\left\{
	\begin{array}{cc}
		\partial_t\varrho+w\cdot \nabla\varrho=g & t>0,\quad \text{in } \Omega,\\
		\varrho(0,\cdot)=\varrho^0(\cdot)& \text{in } \Omega. 
	\end{array}
	\right.
\end{equation}
The following proposition  is proved in \cite{Solonnikov1980}.

\begin{Proposition}
	\label{masseq}
	Let $w$ be a tangent vector field, namely $w\cdot n=0$ on $\partial\Omega$.
	Let $w\in L^1(0,T;[C^{0,1}(\Omega)]^3)$ suppose that
	$$
	\varrho^0\in L^q(\Omega),\quad g\in L^1(0,T;L^q(\Omega)), \quad (q\geq1). 
	$$
	Then, there exists a unique solution $\varrho\in C([0,T];L^q(\Omega))$ to the system \eqref{mass} such that
	\begin{equation}
		\label{1532}
\|\varrho(t)\|_{L^q(\Omega)}\leq e^{\frac{1}{q}\int_0^t\|\nabla\cdot w(s)\|_{L^\infty(\Omega)} \ ds}\left( \|\varrho^0\|_{L^q(\Omega)}+\int_0^te^{-\frac{1}{q}\int_0^s\|\nabla\cdot w(\tau)\|_{L^\infty(\Omega)}\ d\tau}\|g(s)\|_{L^q(\Omega)} \ ds\right). 
	\end{equation}
	Moreover, if
	$$
	\varrho^0\in W^{1,q}(\Omega),\quad g\in L^1(0,T;W^{1,q}(\Omega)), \quad (q\geq1), 
	$$
	then $\varrho\in C([0,T];W^{1,q}(\Omega))$ and 
	\begin{equation}
		\label{1604}
\|\nabla\varrho(t)\|_{[L^q(\Omega)]^3}\leq e^{2\int_0^t\|\nabla w(s)\|_{L^\infty(\Omega)} \ ds}\bigg( \|\nabla\varrho^0\|_{[L^q(\Omega)]^3}+\int_0^te^{-2\int_0^s\|\nabla w(\tau)\|_{L^\infty(\Omega)}\ d\tau}\|\nabla g(s)\|_{[L^q(\Omega)]^3} \ ds\bigg). 
	\end{equation}
    If $q>3$, we have the following inequality
    
\begin{equation}
\label{1511}
\|\varrho(t) - \varrho^0\|_{L^\infty(\Omega)}\leq CT^\alpha\|\nabla \varrho^0\|_{L^\infty(0,T;L^q(\Omega))} + \|g\|_{L^1(0,T;L^\infty(\Omega))},\quad \alpha = 1-\frac{3}{q}.
\end{equation}
\end{Proposition}
 
\begin{proof} 
Let us only prove the last estimate that is crucial for the fixed point part. We define the Lagrangian coordinates associated to the flow $w$ by $\partial_t Y(t,z) = w(t,Y(t,z))$ and $Y(0,z) = z$.  We see that
$$
\widetilde{\varrho}(t,z)-\widetilde{\varrho}(0,z) = \int_0^t\widetilde{g}(s,z)\ ds,
$$
where $\widetilde{\varrho}(t,z)=\varrho(t,Y(t,z))$ and $\widetilde{g}(t,z) = g(t,Y(t,z))$. Then, taking $z=Z(t,y)$, we get
$$
\varrho(t,y)-\varrho^0(y) = \varrho^0(Z(t,y))-\varrho^0(y)+\int_0^t g(s,Z(t,y))\ ds.
$$
Then
$$
\|\varrho(t) - \varrho^0\|_{L^\infty(\Omega)}\leq \|\varrho^0(Z)-\varrho^0\|_{L^\infty(0,T;L^\infty(\Omega))}+\|g\|_{L^1(0,T;L^\infty(\Omega))}.
$$
Since $q>3$, we use Morrey inequality to get
$$
\|\varrho^0(Z)-\varrho^0\|_{L^\infty(0,T;L^\infty(\Omega))}\leq \|Z-\mathbb{I}\|_{L^\infty(0,T;[L^\infty(\Omega)]^3)}^\alpha\|\nabla \varrho^0\|_{L^\infty(0,T;L^q(\Omega))},\quad \alpha = 1-\frac{3}{q}.
$$
Also, we have
$$
Y(t,z)-z = \int_0^t w(s,Y(s,z))\ ds.
$$
Taking $z= Z(t,y)$ and since $w\in L^\infty(0,T;L^\infty(\Omega))$, we obtain
$$
\|Z-\mathbb{I}\|_{L^\infty(0,T;[L^\infty(\Omega)]^3)}\leq T\|w\|_{L^\infty(0,T;[L^\infty(\Omega)]^3)}.
$$
Finally, we get \eqref{1511}. 
\end{proof}

\begin{Lemma}
\label{1441}
Let $1<q<\infty$ and  $f\in L^q(\mathbb{R}^3)$ be given. Then, we have
$(I-\Delta)^{-1/2}\partial_{y_\ell} f\in L^q(\mathbb{R}^3)$ and for any $\varphi\in C_0^\infty(\mathbb{R}^3)$, the following identity holds
\begin{equation}\label{IPP}
\langle \partial_{y_\ell} f, \varphi\rangle
= \langle (I-\Delta)^{-1/2}\partial_{y_\ell} f, \varphi\rangle_{\mathbb{R}^3}
= -\int_{\mathbb{R}^3} f\,(I-\Delta)^{-1/2}\partial_{y_\ell} \varphi \, dy .
\end{equation}
   
\end{Lemma}
\begin{proof}
First, we notice that the operator $(I-\Delta)^{-1/2}\partial_{y_\ell}$ is the Fourier multiplier
with symbol
\[
m(\xi) = \frac{i\xi_\ell}{(1+|\xi|^2)^{1/2}} .
\]
This symbol is smooth, bounded by $1$, and satisfies
$$|\partial^\alpha m(\xi)|\le C_\alpha |\xi|^{-|\alpha|}
$$
for $|\alpha|\le 2$ and $\xi\neq 0$. 
Then, by the Hörmander-Mikhlin multiplier theorem
(see \cite[Ch.~IV, \S3, Thm.~3, p.96]{zbMATH03329342}), this operator extends to a bounded operator on $L^q(\mathbb{R}^3)$ for
$1<q<\infty$:
\begin{equation}\label{eq:mikhlin}
\|(I-\Delta)^{-1/2}\partial_{y_\ell} f\|_{L^q(\mathbb{R}^3)}\le C_q\|f\|_{L^q(\mathbb{R}^3)}.
\end{equation}
In particular $(I-\Delta)^{-1/2}\partial_{y_\ell} f\in L^q(\mathbb{R}^3)$. Then the pairing is defined :
$$\langle \partial_{y_\ell}f,\varphi\rangle:=\int_{\mathbb{R}^3}((I-\Delta)^{-1/2}\partial_{y_\ell} f)\,\varphi\,dy.$$
To justify the integration by parts, we take $f\in\mathscr{S}(\mathbb{R}^3)$. By Plancherel's formula, we get
\[
\int_{\mathbb{R}^3}((I-\Delta)^{-1/2}\partial_{y_\ell} f)\,\varphi\,dy=(2\pi)^{-3}\int_{\mathbb{R}^3} m(\xi)\hat f(\xi)\hat \varphi(-\xi)\,d\xi .
\]
Since $m$ is odd, the change of variable $\xi\mapsto-\xi$ gives
\[
\int_{\mathbb{R}^3}((I-\Delta)^{-1/2}\partial_{y_\ell} f)\,\varphi\,dy=-\int_{\mathbb{R}^3} f\,(I-\Delta)^{-1/2}\partial_{y_\ell} \varphi\,dy .
\]
Finally, by the continuity property of the operator $(I-\Delta)^{-1/2}\partial_{y_\ell}$ in $L^q(\mathbb{R}^3)$ and the density of the Schwartz space $\mathscr{S}(\mathbb{R}^3)$ in $L^q(\mathbb{R}^3)$, we obtain \eqref{IPP}.
\end{proof}

We  also have  the following result which is proved in \cite[Lemma 7.3]{zbMATH06298655}.
\begin{Lemma}
	\label{K.}
	Let $f$ and $g$ be two functions in $L^q(\Omega)$ such that $q>3$ and $p>1$. Then
	\begin{equation}
		\|(\mathbb{I}-\Delta)^{-1/2}\imath (f(t)g(t))\|_{L^q(\Omega)}\leq C\|f(t)\|_{L^q(\Omega)}\|g(t)\|_{L^q(\Omega)}. 
	\end{equation}
	Moreover, if $f$ and $g$ are regular enough, we further have
	\begin{multline}
		\|\partial_t\left[ (\mathbb{I}-\Delta)^{-1/2}\imath (fg)\right] \|_{L^p(0,T;L^q(\Omega))}\leq C \bigg[\left( \int_0^T \|\partial_tf\|^p_{L^q(\Omega)}\|g\|^p_{L^q(\Omega)} \ dt\right)^{1/p}\\ +\left( \int_0^T \|f\|^p_{L^q(\Omega)}\|\partial_tg\|^p_{L^q(\Omega)} \ dt\right)^{1/p} \bigg], 
	\end{multline}
	and 
	\begin{multline}
		\label{1550}
		\|\partial_t\left[ (\mathbb{I}-\Delta)^{-1/2}\imath (\nabla fg)\right] \|_{L^p(0,T;L^q(\Omega))}\leq C \bigg[\left( \int_0^T \|\partial_tf\|^p_{L^q(\Omega)}\|g\|^p_{W^{1,q}(\Omega)} \ dt \right)^{1/p} \\+\left( \int_0^T \|\nabla f\|^p_{L^q(\Omega)}\|\partial_tg\|^p_{L^q(\Omega)} \ dt\right)^{1/p} \bigg] , 
	\end{multline}
	where $C$ is a  positive constant that does not depend on $T$. 
\end{Lemma}

\section{Calculations for the uniqueness part}\label{Apndx-Calcul}

\label{calcul}
This section is devoted to write some calculations regarding the uniqueness part in subsection \ref{Sec-Uniquness}. We recall that $\mathbb{T}_X$ and $\mathbb{N}_X$ are defined in \eqref{defT} and \eqref{defN} respectively. We compute that
\begin{equation*}
	\frac{1}{\rho^2}\mathbb{T}_{X^2}(\rho^2,v^2,\eta^2)-\frac{1}{\rho^1}\mathbb{T}_{X^1}(\rho^1,v^1,\eta^1)= \frac{1}{\rho^2}\mathbb{T}_{X^2}(\rho^2,\overline{v},\eta^2)+\mathscr{T}_1+\mathscr{T}_2+\mathscr{T}_3+\mathscr{T}_4+\mathscr{T}_5+\mathscr{T}_6+\mathscr{T}_7,
\end{equation*}
with
$$
\mathscr{T}_1=(\nabla X^2-\nabla X^1)\partial_t v^1,\quad \mathscr{T}_2=(\nabla X^2-\nabla X^1)\nabla v^1\partial_t Y^2(X^2),\quad \mathscr{T}_3=\nabla X^1\nabla v^1(\partial_t Y^2(X^2)-\partial_t Y^1(X^1)),
$$
$$\mathscr{T}_4=\nabla\cdot ((\partial_t X^1-\partial_t X^2)\otimes v^1),\quad \mathscr{T}_5=-(\nabla\cdot v^1)(\partial_t X^1-\partial_t X^2),\quad
\mathscr{T}_6= -\partial_t(\delta_{X^2}-\delta_{X^1})b_{X^1}^{-\top} v^1,
$$
$$
\mathscr{T}_7= -\partial_t\delta_{X^1}(b_{X^2}^{-\top}-b_{X^1}^{-\top}) v^1,
$$
where we have used 
$$
\nabla(\partial_t X^1-\partial_t X^2)v^1=\nabla\cdot ((\partial_t X^1-\partial_t X^2)\otimes v^1)-(\nabla\cdot v^1)(\partial_t X^1-\partial_t X^2).
$$
On the other hand, we have for the convection term 
\begin{equation*}
	\frac{1}{\rho^2}\mathbb{N}_{X^2}(\rho^2,v^2,\eta^2)-\frac{1}{\rho^1}\mathbb{N}_{X^1}(\rho^1,v^1,\eta^1)= \mathscr{C}_{v^2}(\overline{v})+ \mathscr{N}_1+\mathscr{N}_2+\mathscr{N}_3+\mathscr{N}_4+\mathscr{N}_5+\mathscr{N}_6+\mathscr{N}_7,
\end{equation*}
where $\mathscr{C}_{v^2}(\overline{v})$ corresponds to the converction term of the velocity $v^2$ that is given by
	\begin{equation}
		\label{1020}
		\mathscr{C}_{v^2}(\overline{v})=\nabla\cdot\left( (b_{X^2}^{-\top}v^2\otimes b_{X^2}^{-\top}\overline{v}) b_{X^2}\right)-(\nabla\cdot v^2)b_{X^2}^{-\top}(v^2-v^1),
\end{equation}
and the nonlinear terms are given by
$$
\mathscr{N}_1=\nabla\cdot [((b^{-\top}_{X^2}v^2-b^{-\top}_{X^1}v^1)\otimes b^{-\top}_{X^2}v^1)b_{X^2}],\quad \mathscr{N}_2=\nabla\cdot [(b^{-\top}_{X^1}v^1\otimes (b^{-\top}_{X^2}v^1-b^{-\top}_{X^1}v^1))b_{X^2}],
$$
$$\mathscr{N}_3=\nabla\cdot[ (b^{-\top}_{X^1}v^1\otimes b^{-\top}_{X^1}v^1)(b_{X^2}-b_{X^1})],\quad \mathscr{N}_4= -\nabla\cdot\left( b^{-\top}_{X^2}v^1 \otimes(v^2-v^1)\right),
$$
$$
\mathscr{N}_5=((v^2-v^1)\cdot \nabla) (b^{-\top}_{X^2}v^1),\quad \mathscr{N}_6=-(\nabla \cdot v^1)(b_{X^2}^{-\top}-b_{X^1}^{-\top})v^1,
$$
where we have used
$$
(\nabla\cdot (v^2-v^1) )b^{-\top}_{X^2}v^1=\nabla\cdot\left( b^{-\top}_{X^2}v^1 \otimes(v^2-v^1)\right) - ((v^2-v^1)\cdot \nabla) (b^{-\top}_{X^2}v^1).
$$
Now, we deal with the diffusion term
$$
\frac{1}{\rho^2}\nabla\cdot\mathbb{S}_{X^2}(v^2)-\frac{1}{\rho^1}\nabla\cdot\mathbb{S}_{X^1}(v^1)=\frac{1}{\rho^2}\nabla\cdot \mathbb{S}_{X^2}(\overline{v})+\frac{1}{\rho^2}\nabla\cdot \mathscr{S}_1+\mathscr{S}_2,
$$
where
$$
\mathscr{S}_2=-\frac{\overline{\rho}}{\rho^1\rho^2}\nabla\cdot\mathbb{S}_{X^1}(v^1),
$$
and
\begin{align*}
	\left[\mathscr{S}_1\right]_{ij}
	&=(\mathbb{S}_{X^2}(v^1))_{ij}-(\mathbb{S}_{X^1}(v^1))_{ij}
    \\
    &=\mu \bigg(
	\sum_{m,k}\left(\frac{\partial (b_{X^2}^{-\top})_{ik}}{\partial x_m}(X^2)
	-\frac{\partial (b_{X^1}^{-\top})_{ik}}{\partial x_m}(X^1)\right) (v^1)_k (b_{X^2})_{mj}
    \\
    & \ +\left(\frac{\partial (b_{X^2}^{-\top})_{mk}}{\partial x_i}(X^2) -\frac{\partial (b_{X^1}^{-\top})_{mk}}{\partial x_i}(X^1) 
	\right) (v^1)_k (b_{X^2})_{mj}
	\\
	& \ +
	\left(\frac{\partial (b_{X^1}^{-\top})_{ik}}{\partial x_m}(X^1)+\frac{\partial (b_{X^1}^{-\top})_{mk}}{\partial x_i}(X^1)\right)(v^1)_k\left( (b_{X^2})_{mj}-(b_{X^1})_{mj} \right)  
	\\
    & \ +
	\sum_{m,k,l}\bigg( 
	(b_{X^2}^{-\top})_{ik}(X^2)\frac{\partial (Y^2)_l}{\partial x_m}(X^2)-(b_{X^1}^{-\top})_{ik}(X^1)\frac{\partial (Y^1)_l}{\partial x_m}(X^1)
	\\ 
	& \ +(b_{X^2}^{-\top})_{mk}(X^2)\frac{\partial (Y^2)_l}{\partial x_i}(X^2)-(b_{X^1}^{-\top})_{mk}(X^1)\frac{\partial (Y^1)_l}{\partial x_i}(X^1)
	\bigg) \frac{\partial (v^1)_k}{\partial y_l}
	(b_{X^2})_{mj}
    \\
    & \ + \left((b_{X^1}^{-\top})_{ik}(X^1)\frac{\partial (Y^1)_l}{\partial x_m}(X^1)+ (b_{X^1}^{-\top})_{mk}(X^1)\frac{\partial (Y^1)_l}{\partial x_i}(X^1)\right)\frac{\partial (v^1)_k}{\partial y_l}
	\left( (b_{X^2})_{mj}-(b_{X^1})_{mj}\right)  \bigg)
	\\
	& \ +\lambda\nabla\cdot v^1\left( \frac{1}{\delta_{X^2}}-\frac{1}{\delta_{X^1}}\right) (b_{X^2})_{i,j}+\lambda\nabla\cdot v^1\frac{1}{\delta_{X^1}}\left( (b_{X^2})_{i,j}-(b_{X^1})_{ij}\right).
\end{align*}
We get for the pressure term
$$
\frac{1}{\rho^2}\nabla\cdot((\rho^2)^{\gamma}b_{X^2}^\top)-\frac{1}{\rho^1}\nabla\cdot((\rho^1)^{\gamma}b_{X^1}^\top)= \mathscr{P}_1+ \mathscr{P}_2+\mathscr{P}_3+\mathscr{P}_4,
$$
with
$$
\mathscr{P}_1= \nabla\cdot \left(\frac{1}{\rho^2}((\rho^2)^{\gamma}(b_{X^2}^\top-b_{X^1}^\top))\right),\quad \mathscr{P}_2= -((\rho^2)^{\gamma}(b_{X^2}^\top-b_{X^1}^\top))\nabla\left(\frac{1}{\rho^2}\right),
$$
$$
\mathscr{P}_3=\frac{1}{\rho^2}\nabla\cdot (((\rho^2)^\gamma-(\rho^1)^\gamma)b_{X^1}^{-\top}),\quad \mathscr{P}_4=-\frac{\overline{\rho}}{\rho^1\rho^2}\nabla\cdot ((\rho^1)^\gamma b_{X^1}^{-\top}),
$$
where we have used
$$
\frac{1}{\rho^2}\nabla\cdot((\rho^2)^{\gamma}(b_{X^2}^\top-b_{X^1}^\top))=\nabla\cdot \left(\frac{1}{\rho^2}((\rho^2)^{\gamma}(b_{X^2}^\top-b_{X^1}^\top))\right)-((\rho^2)^{\gamma}(b_{X^2}^\top-b_{X^1}^\top))\nabla\left(\frac{1}{\rho^2}\right) . 
$$
For the structure term, we obtain
$$
-(\mathbb{S}_{X^2}(v^2)-\mathbb{S}_{X^1}(v^1))+(\rho^2)^\gamma-(\rho^1)^{\gamma}=-(\mathbb{S}_{X^2}(\overline{v})+\mathscr{S}_1)N \cdot e_3+(\rho^2)^\gamma-(\rho^1)^{\gamma}.
$$
In the sequel, we will use the notations
$$
N^i=N_{\eta^i},\quad i=1,2.
$$
The boundary terms write
\begin{multline}
	\left[\frac{|N|}{|N^2|}b_{X^2}^{-1}\mathbb{D}_{X^2}(v^2)-\frac{|N|}{|N^1|}b_{X^1}^{-1}\mathbb{D}_{X^1}(v^1)\right]n\cdot\tau_j=\bigg[ \frac{|N|}{|N^2|}b_{X^2}^{-1}\mathbb{D}_{X^2}(\overline{v})n+\frac{|N|}{|N^2|}b_{X^2}^{-1}\left(\mathscr{S}_1\right)\bigg] \cdot \tau^j\\+[\mathscr{K}_1]_j+[\mathscr{K}_2]_j,
\end{multline}
where
$$
[\mathscr{K}_1]_j=\frac{|N|}{|N^2|}(b_{X^2}^{-1}-b_{X^1}^{-1} )\mathbb{D}_{X^1}(v^1)n\cdot\tau_j,\quad [\mathscr{K}_2]_j=|N|\left( \frac{1}{|N^2|}-\frac{1}{|N^1|}\right)b_{X^1}^{-1} \mathbb{D}_{X^1}(v^1)n\cdot\tau_j.
$$
We have also the friction term
$$
\beta_1 b_{X^2}^{-1}(b_{X^2}^{-\top}v^2-\partial_t\eta^2 e_3)- \beta_1 b_{X^1}^{-1}(b_{X^1}^{-\top}v^1-\partial_t\eta^1 e_3)  \cdot \tau^j= \beta_1 b_{X^2}^{-1}(b_{X^2}^{-\top}\overline{v}-\partial_t\overline{\eta} e_3)\cdot \tau^j+[\mathscr{K}_3]_j+[\mathscr{K}_4]_j,
$$
where
$$
[\mathscr{K}_3]_j=\beta_1(b^{-1}_{X^2}b^{-\top}_{X^2}-b^{-1}_{X^1}b^{-\top}_{X^1})v^1,\quad [\mathscr{K}_4]_j=-\beta_1(b^{-1}_{X^2}-b^{-1}_{X^1})\partial_t\eta^1e_3.
$$
\\
Then, the difference $(\overline{v},\overline{\eta})$ satisfies the following linear system
\begin{equation}
	\label{1336}
	\left\{
	\begin{array}{cc}
		\left( \mathbb{T}_{X^2}(\rho^2,\overline{v},\eta^2)+\rho^2\mathscr{C}_{v^2}(\overline{v})-\nabla\cdot (\mathbb{S}_{X^2}(\overline{v}))\right)  =\mathscr{F}_1+\nabla\cdot\mathscr{F}_2& \text{ in } (0,T)\times \Omega,  \vspace{0.15cm}\\
		\partial_{tt}\overline{\eta} + \Delta^2 \overline{\eta} -\Delta \partial_t\overline{\eta}= -(\mathbb{S}_{X^2}(\overline{v})) N\cdot e_3 -\mathscr{F}_2N\cdot e_3& \text{in } (0,T)\times\omega,
	\end{array}
	\right.
\end{equation}
and the boundary conditions
\begin{equation}
	\label{1335}
	\left\{
	\begin{array}{cc}
		(\overline{v})_n = 0,\   \left[2\mu\mathbb{D}(\overline{v})n +\beta_0 \overline{v}\right] \cdot \tau^j=0 &\text{ on } (0,T)\times\Gamma_{0},\\
		\left(\overline{v}-\partial_t \overline{\eta} e_3 \right)_{n}=0,\  \left[ 2\mu b_{X^2}^{-1}\frac{|N|}{|N^2|}\mathbb{D}_{X^2}(\overline{v})n+\mathscr{F}_2n+\beta_1 b_{X^2}^{-1}(b_{X^2}^{-\top}\overline{v}-\partial_t\overline{\eta} e_3) \right] \cdot \tau^j=[ \mathscr{K}]_j &\text{ on } (0,T)\times\Gamma (\eta^0),
	\end{array}
	\right.
\end{equation}
where
\begin{equation}
	\label{1443}
	\mathscr{F}_1=\mathscr{F}_{11}+\mathscr{F}_{12},
\end{equation}
with
\begin{equation}
	\label{1531}
	\mathscr{F}_{11}= -\rho^2 (\mathscr{T}_4+\mathscr{N}_1+\mathscr{N}_2+\mathscr{N}_3+\mathscr{N}_4+\mathscr{P}_1)=-\rho^2\nabla\cdot \mathscr{F},
\end{equation}
and that 
\begin{multline}
	\label{1806}
	\mathscr{F}=\Big[ ((b^{-\top}_{X^2}v^2-b^{-\top}_{X^1}v^1)\otimes b^{-\top}_{X^2}v^1)b_{X^2}+(b^{-\top}_{X^1}v^1\otimes (b^{-\top}_{X^2}v^1-b^{-\top}_{X^1}v^1))b_{X^2}\\+ (b^{-\top}_{X^1}v^1\otimes b^{-\top}_{X^1}v^1)(b_{X^2}-b_{X^1})+\left(\frac{1}{\rho^2}((\rho^2)^{\gamma}(b_{X^2}^\top-b_{X^1}^\top))\right)\\-\left( b^\top_{X^2}v^1 \otimes(v^2-v^1)\right)+((\partial_t X^1-\partial_t X^2)\otimes v^1)\Big],
\end{multline}
and 
\begin{align}
	\label{1239}
	\mathscr{F}_{12} & =-\rho^2(\mathscr{T}_1+\mathscr{T}_2+\mathscr{T}_3+\mathscr{T}_5+\mathscr{T}_6+\mathscr{T}_7+\mathscr{N}_5+\mathscr{N}_6-\mathscr{S}_2+\mathscr{P}_4+\mathscr{P}_2)
    \notag \\
    &=-\rho^2\Big[(\nabla X^2-\nabla X^1) \partial_t v^1+(\nabla X^2-\nabla X^1)\nabla v^1\partial_t Y^2 +\nabla X^1\nabla v^1(\partial_t Y^2-\partial_tY^1)
    \notag \\
	& \ -\partial_t(\delta_{X^2}-\delta_{X^1})b^{-\top}_{X^1} v^1-\partial_t\delta_{X^1}(b_{X^2}^{-\top}-b_{X^1}^{-\top})v^1
	-(\nabla\cdot v^1)(b^\top_{X^2}-b^\top_{X^1})v^1
    \notag\\
    &\ +(\nabla\cdot v^1)(b^\top_{X^1})(v^2-v^1) +\frac{\overline{\rho}}{\rho^1\rho^2}\nabla\cdot\left(\mathbb{S}_{X^1}(v^1)-(\rho^1)^\gamma b_{X^1}^\top\right)+((v^2-v^1)\cdot \nabla) (b^\top_{X^2}v^1)\notag\\
    & \ -(\nabla\cdot v^1)(\partial_t X^1-\partial_t X^2)+((\rho^2)^{\gamma}(b_{X^2}^\top-b_{X^1}^\top))\nabla\left( \frac{1}{\rho^2}\right)\Big].
\end{align}
The remaining terms in the fluid equation write as a divergence term
\begin{equation}
	\label{1235}
	\nabla\cdot\mathscr{F}_2=\nabla\cdot\mathscr{S}_1-\mathscr{P}_3=\nabla\cdot(\mathscr{S}_1+((\rho^1)^{\gamma}-(\rho^2)^{\gamma})b_{X^1}^\top)),
\end{equation}
with
\begin{equation}
	\label{1425}
	\mathscr{F}_2=\mathscr{S}_1+((\rho^1)^{\gamma}-(\rho^2)^{\gamma})b_{X^1}^\top).
\end{equation}
Finally, the boundary term is given by
\begin{multline}
	\label{2010}
	[\mathscr{K}]_j=-[\mathscr{K}_1]_j-[\mathscr{K}_2]_j-[\mathscr{K}_3]_j-[\mathscr{K}_4]_j\\=-\bigg(\frac{|N|}{|N^2|}(b_{X^2}^{-1}-b_{X^1}^{-1} )\mathbb{D}_{X^1}(v^1)n+|N|\left( \frac{1}{|N^2|}-\frac{1}{|N^1|}\right)b_{X^1}^{-1} \mathbb{D}_{X^1}(v^1)n\\+\beta_1(b^{-1}_{X^2}b^{-\top}_{X^2}-b^{-1}_{X^1}b^{-\top}_{X^1})v^1+\beta_1(b^{-1}_{X^2}-b^{-1}_{X^1})(v^1-\partial_t\eta^1e_3)\bigg) \cdot \tau^j, \ \ \ j=1,2.
\end{multline}

\section{Weak regularity estimates}\label{Apndx-Weak-regu}
In this section, we derive an energy estimate for the system \eqref{1336-1}--\eqref{1335-1}. After an appropriate change of variables, the system can be reformulated as a linearly coupled problem posed on a moving domain with prescribed velocity $v^\star$. Using classical arguments, including the Reynolds transport theorem, we first obtain a weak estimate in this moving framework. We then return to fixed coordinates to recover the desired estimate in a stationary configuration.

\vspace*{0.5cm}
\noindent
We recall that the definition of $\mathbb{T}_{X_{\eta^\star}}$, $\mathbb{S}_{X_{\eta^\star}}$, $\mathscr{C}_{v^\star}$ are given in \eqref{defT}, \eqref{defS} and \eqref{1020} respectively.
\begin{Lemma}
	\label{1821}
	Let 
	$$f\in L^2(0,T;[(H^1(\Omega))']^3),\quad F\in L^2(0,T;[L^2(\Omega)]^9),\quad k\in L^2(0,T;[H^{-1/2}(\Gamma(\eta^0)))]^3).$$
	 Now, for given 	
	$$
	\eta^\star\in L^p(0,T;W^{4,q}(\omega))\cap W^{2,p}(0,T;L^q(\omega)),\quad 
	v^\star\in L^p(0,T;[W^{2,q}(\Omega)]^3)\cap W^{1,p}(0,T;[L^q(\Omega)]^3),
	$$
	with $p>2$ and $q>3$,  consider $(v,\eta)$ the solution of the linear system 
 \begin{equation}
		\label{1028}
		\left\{
		\begin{array}{cc}
		\mathbb{T}_{X_{\eta^\star}}(\rho^\star, v,\eta^\star)+\rho^\star\mathscr{C}_{v^\star}(v)-\nabla\cdot (\mathbb{S}_{X_{\eta^\star}}(v)) =f+\nabla\cdot F& \text{ in } (0,T)\times \Omega,  \vspace{0.15cm}\\
			\partial_{tt}\eta + \Delta^2 \eta -\Delta\partial_t\eta=-(\mathbb{S}_{X_{\eta^\star}}(v)) N\cdot e_3 - FN\cdot e_3& \text{in } (0,T)\times\omega,
		\end{array}
		\right.
	\end{equation}
	with the boundary conditions
	\begin{equation}
		\label{1021}
		\left\{
        \begin{array}{cc}
		v_n = 0,\quad \left[2\mu\mathbb{D}(v)n +\beta_0 v\right]_\tau=0 &\text{ on } (0,T)\times\Gamma_{0},\vspace{0.15cm}\\ 
			\left( v-\partial_t \eta e_3 \right)_{n}=0, &\text{ on } (0,T)\times\Gamma (\eta^0),\\ \left[ b_{X_{\eta^\star}}^{-1}\frac{|N|}{|N_{\eta^\star}|}\left( 2\mu\mathbb{D}_{X_{\eta^\star}}( v)+F\right) n+\beta_1 b_{X_{\eta^\star}}^{-1}(b_{X_{\eta^\star}}^{-\top} v-\partial_t\eta e_3) \right]_\tau = k&\text{ on } (0,T)\times\Gamma (\eta^0),
		\end{array}
		\right.
	\end{equation}
	and zero initial conditions, where   $$\rho^\star \in L^\infty(0,T;W^{1,q}(\Omega))\cap W^{1,p}(0,T;L^q(\Omega))$$ is the unique solution to the 
     following linearized system 
     \begin{align*}
     \left\{
     \begin{array}{cc}
\partial_t\rho^\star + \frac{1}{\delta_{X^\star}} \left( v^\star - \partial_t X^\star \right) \nabla \rho^\star = -\frac{\rho^\star}{\delta_{X^\star}} \nabla \cdot v^\star  \ \ \text{ in } (0,T)\times \Omega, \\ \rho^\star(0)=\rho^0 \ \  \text{ in } \Omega.
\end{array}
\right. 
     \end{align*}
     
    Then, we have the following estimate
	\begin{multline}
		\label{1022}
		\|v\|_{L^2(0,T;[H^1(\Omega)]^3)}+ \|v\|_{L^\infty(0,T;[L^2(\Omega)^3])}+ \|\eta\|_{L^\infty(0,T;H^2(\omega))}+\|\partial_t\eta\|_{L^\infty(0,T;L^2(\omega))}\\+\|\partial_t\eta\|_{L^2(0,T;H^1(\omega))} \leq C\bigg(  \|f\|_{L^2(0,T;[(H^1(\Omega))']^3)}+ \|F\|_{L^2(0,T;[L^2(\Omega)]^9)} + \|k\|_{L^2(0,T;[W^{-1/2,2}(\Gamma(\eta^0))]^3)}\bigg),
	\end{multline}
	 with  $C>0$, a constant independent of T and dependent only on $\|\eta^\star\|_{L^\infty(0,T;W^{3,q}(\omega))}$.
\end{Lemma}
\begin{proof}	
	We note that the system \eqref{1028}--\eqref{1021} is equivalent to the linear problem written in the moving domain $\Omega(\eta^\star)$ using the change of coordinates $(X^\star)^{-1}=X^{-1}_{\eta^\star}$:
	\begin{equation}
		\label{1711}
		\left\{
		\begin{array}{cc}
	\varrho^\star(\partial_t w+(\vartheta^\star\cdot \nabla )w)-\nabla\cdot\mathbb{S}(w)=f^\star+\nabla\cdot F^\star& \text{ in } (0,T)\times \Omega(\eta^\star),  \vspace{0.15cm}\\
			\partial_{tt}\eta + \Delta^2 \eta - \Delta\partial_t\eta=-(\mathbb{S}(w)) N_{\eta^\star}\cdot e_3 - F^\star N_{\eta^\star}\cdot e_3& \text{in } (0,T)\times\omega,
		\end{array}
		\right.
	\end{equation}
	with zero initial conditions and the  boundary conditions
	\begin{equation}
		\label{1712}
		\left\{
		\begin{array}{cc}
		w_{n_{\eta^\star}} = 0,\quad \quad \left[2\mu\mathbb{D}(w)n_{\eta^\star} +F^\star n_{\eta^\star}+\beta_0 w\right]_{\tau_{\eta^\star}}=0 &\text{ on } (0,T)\times\Gamma_{0},  \vspace{0.15cm}\\
			\left( w-\partial_t \eta e_3 \right)_{n_{\eta^\star}}=0& \text{ on } (0,T)\times\Gamma (\eta^\star),  \vspace{0.15cm}\\ \left[2\mu \mathbb{D}(w)n_{\eta^\star}+F^\star n_{\eta^\star}+\beta_1 (w-\partial_t\eta e_3) \right]_{\tau_{\eta^\star}} = k^\star & \text{ on } (0,T)\times\Gamma (\eta^\star),
		\end{array}
		\right.
	\end{equation}
where
	$$
	 w=b_{X_\eta^\star}^{-\top} v((X^\star)^{-\top}(\cdot)),\quad f^\star=\frac{1}{\delta_{X^\star}}f((X^{\star})^{-1}(\cdot)),\quad F^\star=F((X^{\star})^{-1}(\cdot))b_{X^\star}^{-1},\quad k^\star=b_{X^\star}k((X^{\star})^{-1}(\cdot)),
	$$
    and $\varrho^\star = \rho^\star((X^\star)^{-\top}(\cdot))$, $\vartheta^\star = b_{X_\eta^\star}^{-\top} v^\star((X^\star)^{-\top}(\cdot))$ satisfy
\begin{align*}
\partial_t\varrho^\star+\nabla\cdot (\varrho^\star \vartheta^\star)=0& \text{ in } (0,T)\times \Omega(\eta^\star), 
\end{align*}
along with given  initial datum.

Now, we multiply the system by $(w,\partial_t\eta)$ and using Reynolds transport formula, we see that the convective term disappears. In the other hand, since $\eta^\star\in L^\infty(0,T;W^{3,q}(\omega))$, the domain $\Omega(\eta^\star(t))$ is in particular $C^{1,1}$, we can then use the coercivity property of the Lamé operator with Navier boundary conditions that is proved in \cite[Lemma 2.3]{zbMATH05054501}, we obtain
\begin{multline}
\frac{1}{2}\|\sqrt{\varrho^\star(t)}w(t)\|_{[L^2(\Omega(\eta^\star(t))]^3)}^2 + \|w\|^2_{L^2(0,t;[H^1(\Omega(\eta^\star))]^3)}+\|\nabla\partial_t\eta\|^2_{L^2(0,t;[L^2(\omega)]^2)}+\|\Delta\eta(t)\|^2_{L^2(\omega)} +\|\partial_t\eta(t)\|^2_{L^2(\omega)} \\
\leq C\bigg(  \|f^\star\|^2_{L^2(0,t;[(H^1(\Omega(\eta^\star)))']^3)}+ \|F^\star\|^2_{L^2(0,t;[L^2(\Omega(\eta^\star))]^9)} + \|k^\star\|^2_{L^2(0,t;[W^{-1/2,2}(\Gamma(\eta^\star))]^3)}\bigg).
\end{multline}
Using the change of variables $X^\star$, then estimate \eqref{1022} follows.
\end{proof}

    \bibliographystyle{plain}
	\bibliography{biblio}

\end{document}